%% file: EqAvoiders.tex
\documentclass[a4paper,twoside]{amsart}
\usepackage[pdftex]{color,graphicx} 


\include{PreambleHR2010}

\include{EnglishEnvironments}

\usepackage[round]{natbib}

\usepackage[pdftex]{hyperref}

\begin{document}

\title[Equivalence classes of permutations avoiding a pattern]{
Equivalence classes of permutations avoiding a pattern}

\author{Henning Arn\'or \'Ulfarsson}

\thanks{This work is supported by grant no.\ 090038011 from the Icelandic Research Fund.}

\address{School of Computer Science, Reykjav\'ik University,
Nauth\'olsv\'ik, Menntavegi 1, 101 Reykjav\'ik, Iceland}

\email{henningu@ru.is}


\begin{abstract}
Given a permutation pattern $\bfp$ and an equivalence relation on permutations, we study the corresponding equivalence classes all of whose members avoid $\bfp$.  Four relations are studied: Conjugacy, order isomorphism, Knuth-equivalence and toric equivalence.  Each of these produces a known class of permutations or a known counting sequence.  For example, involutions correspond to conjugacy, and permutations whose insertion tableau is hook-shaped with $2$ in the first row correspond to Knuth-equivalence.  These permutations are equinumerous with certain congruence classes of graph endomorphisms.  In the case of toric equivalence we find a class of permutations that are counted by the Euler totient function, with a subclass counted by the number-of-divisors function.  We also provide a new symmetry for bivincular patterns that produces some new non-trivial Wilf equivalences
\end{abstract}

\maketitle
\thispagestyle{empty}

\setcounter{tocdepth}{1}
\tableofcontents

\section{Introduction}
Let $\symS_n$ be the permutation group on $n$ letters. We will mostly use
one-line notation for the elements of this group, i.e., the permutation in
$\symS_4$ that sends $1 \mapsto 2$, $2 \mapsto 4$, $3 \mapsto 1$, $4 \mapsto 3$
will be written $2413$. This corresponds to the bottom line in the two-line notation
\[
\biv{1\bl2\bl3\bl4}{\vinb{2\bl4\bl1\bl3}}.
\]

In usual pattern avoidance and matching one studies and counts the
permutations in $\symS_n$ that avoid or match a particular pattern. These
permutations belong to two sets:
\begin{align*}
\symA_n(\bfp) &= \{ \pi \in \symS_n \sep \pi
\text{ avoids } \bfp  \}, \\
\symM_n(\bfp) &= \{ \pi \in \symS_n \sep \pi
\text{ matches } \bfp  \}.
\end{align*}

For example if the chosen pattern is $\bfp = 123$ (an increasing subsequence of three letters)
we have
\begin{align*}
\symA_4(\bfp) = \{ &1432, 2143, 2413, 2431, 3142, 3214, 3241, \\
& 3412, 3421, 4132, 4213, 4231, 4312, 4321  \}, \\
\symM_4(\bfp) = \{ &\ull{123}4, \ull{124}3, \ull{13}2\ull{4}, \ull{134}2, \ull{1}4\ull{23},
\ull{2}1\ull{34}, \ull{234}1, \\
&  \ull{23}1\ull{4}, 3\ull{124}, 4\ull{123} \},
\end{align*}
where we have marked an occurrence of the pattern by underlining the letters. Note that
we do not care if the pattern appears as a subsequence whose entries
are adjacent or not in the permutation. In general the number of permutations
in $\symS_n$ that avoid $123$ is the $n$-th Catalan number, see e.g., \cite{MR829358},
\[
|\symA_n(\bfp)| = C_n = \frac{1}{n+1}\binom{2n}{n}.
\]
The permutations that match (or contain) the same pattern
are the remaining permutations, giving $|\symM_n(\bfp)| = n! - C_n$. This fact
remains true for any pattern from $\symS_3$.

In this paper we will explore how pattern avoidance interacts
with equivalence relations. More precisely, given an equivalence relation
and a pattern we will study and count the equivalence classes that do not
contain any permutations that match the pattern. Similarly we will look at
equivalence classes that only contain permutations that match the pattern.
In short, we will be studying the two sets
\begin{align*}
\eqAn{\bfp} &= \{ \pi \in \symS_n \sep \pi
\text{ and every equivalent permutation avoids } \bfp  \}, \\
\eqMn{\bfp} &= \{ \pi \in \symS_n \sep \pi
\text{ and every equivalent permutation matches } \bfp  \}.
\end{align*}
If the relation can be extended to patterns we will also study
the sets
\begin{align*}
\Aeqn{\bfp} &= \{ \pi \in \symS_n \sep \pi
\text{ avoids } \bfp \text{ and every equivalent pattern} \}, \\
\Meqn{\bfp} &= \{ \pi \in \symS_n \sep \pi
\text{ matches } \bfp \text{ and every equivalent pattern} \}.
\end{align*}

The motivation for this work came about somewhat by accident when
I was studying Knuth-equivalent permutations and occurrences of patterns
in them. Two permutations are \emph{Knuth-equivalent} if they have the same
\emph{insertion tableaux}, see e.g., \cite{MR1464693}. I was trying to see how the
occurrence of the pattern $231$ could be seen on the tableau and because
of a bug in the code I had written I generated permutations whose entire
class avoids the pattern. So I had generated the set $\eqAn{231}$.
It turns out that the permutations in this set correspond to hook-shaped
tableaux, with $1,2,\dots k$ in the first line, such as
\[
\young(1234,5,6)
\]
Now, the equivalence class of the pattern $231$
is $\{231,213\}$ and it is possible to show that a permutation avoids these two patterns
if and only if it has an insertion tableau of the type described above.
So we now have (see Proposition \ref{prop:orig} and Corollary \ref{cor:knuth-S3stable})
\[
\eqAn{231} = \symA_n(231,213) = \symA_n(\widetilde{231}).
\]
A pattern with this property will be called \emph{stable} below. Since the hook-length
formula can be used to count the number of tableaux of a particular shape it
can be used to count the sets above, giving
\[
|\symA_n(231,213)| = 2^{n-1},
\]
which is well-known, and in fact \cite[Proposition 5.3]{MR2061380} used a similar method to prove
it.

Knuth-equivalence is one of four equivalence relations we will look at here,
see the overview below for some examples of the results.

\subsection{Overview and main results}
The question of how patterns interact with equivalence relations seems to
only become interesting when we consider so-called \emph{bivincular patterns},
which generalize and subsume classical patterns. They are defined in
section \ref{sec:gen}, but roughly speaking they are classical patterns with
additional requirements on what positions and values are allowed in the occurrence
of the pattern.

The next section deals with generalities that can be applied to any equivalence
relation. The main result there is Proposition \ref{prop:basic}, which reduces
the number of patterns one needs to look at. For example, let $\pi^\rmr$ be
the permutation $\pi$ read backwards. If for any permutations
$\pi$ and $\mu$ we have $\pi$ is equivalent to $\mu$ if and only if $\pi^\rmr$
is equivalent to $\mu^\rmr$,
then the proposition shows that $|\eqAn{\bfp}| = |\eqAn{\bfp^\rmr}|$,
$|\eqMn{\bfp}| = |\eqMn{\bfp^\rmr}|$. This means that if we are interested in
counting then there are half as many patterns to look at.

The next four sections deal with a particular equivalence relation each.
The ordering of the sections is based on the difficulty of the proofs.

In section \ref{sec:Conjugacy} we regard two permutations as equivalent
if they have the same cycle type, or equivalently, $\pi \sim \mu$ if
there exists a permutation $\sigma$ such that $\pi = \sigma \mu \sigma^\rmi$.
A particular result from that section is that permutations whose entire class avoids
the bivincular pattern $\biv{\nivls{1\bl2\bl3}}{\vinls{2\bl3}\bl\vinbs{1}}$
are the involutions. (Recall that bivincular patters are defined
in section \ref{sec:gen} below.)

In section \ref{sec:Order} we regard two permutations as equivalent if
they have the same order so the equivalence classes here will be unions
of conjugacy classes. The results in this section will rely on results from
the previous one. Here's an example: permutations whose entire class avoids
$\biv{ \nivls{1\bl2\bl3\bl\dotsm\bl k} }{ \vinls{2\bl3\bl\dotsm\bl k\bl1} }$
are permutations with order that can not be written $k \cdot m$ with $k \nmid m$.

In section \ref{sec:Knuth} we regard two permutations as equivalent if
they are Knuth-equivalent, meaning that they have the same insertion tableau.
For another characterization of this equivalence relation, see Definition
\ref{def:Knuth-eq}. The result mentioned in the introduction is from this section.
Another is that
permutations whose entire class avoids $\biv{\nivls{1}\bl\nivbs{2\bl3}}{\vinbs{2\bl3\bl1}}$
are permutations whose insertion tableau is hook-shaped and has $2$ in the first row.
These permutations are counted with $1 + \frac{1}{2}\binom{2n-2}{n-1}$ and are
equinumerous with certain congruence classes of graph endomorphisms, see
Proposition \ref{prop:Graphs} and Open Problem \ref{openp:Graphs}.
Finally, in Corollary \ref{cor:knuth-S3stable}, we show that every classical pattern
from $\symS_3$ is stable.

In section \ref{sec:Toric} we regard two permutations as equivalent if they
are in the same \emph{toric class}, which is defined rigorously below, but
very roughly speaking this
means that two permutations are equivalent if their permutation matrices
become the same when they are wrapped around a torus. There are three
main results in this section. The first is Proposition \ref{prop:newsym},
which provides a new symmetry
relation for certain bivincular patterns, which produces some new non-trivial
Wilf-equivalences. Then there are two related results,
the first showing that the permutations whose entire class
avoids $\biv{\nivs{1\bl2}\bl\nivrs{3}}{\vinbs{2\bl1\bl3}}$ have a rich
structure and are enumerated by the Euler totient function, $\phi(n+1)$, see
Theorem \ref{thm:totient}; and that the subset of permutations whose entire class avoids
$\biv{\nivs{1\bl2}\bl\nivbs{3}}{\vinbs{2\bl1\bl3}}$
are equinumerous with the divisors of $n$, see Theorem
\ref{thm:divisors}. These two results allow us to state a conjecture that is
equivalent to the Riemann Hypothesis, see Conjecture \ref{conj:eqRH}.

In the last section, \ref{sec:other}, we mention some other equivalence
relations we have considered.

\section{Generalities} \label{sec:gen}

Below we will recall the definition of bivincular patterns, but first recall
that given a pattern $\bfp$ we say that it \emph{occurs} in a permutation
$\pi$, or that $\pi$ \emph{matches} $\bfp$, if $\pi$ contains a subsequence
whose letters are in the same order as the letters in the pattern. If a permutation
does not match a pattern we say that it \emph{avoids} the pattern.
\begin{example}
The permutation $\pi = 241635$ has four occurrences of the pattern $123$,
given by the subsequences $246$, $245$, $135$ and $235$. It has two occurrences
of the pattern $231$, given by the subsequences $241$ and $463$. It avoids
the pattern $321$.
\end{example}

A \emph{vincular} pattern is like a classical pattern but it can have requirements
on which letters must be adjacent in the occurrence. More precisely if two adjacent
letters in the pattern are underlined then the corresponding letters in the occurrence
must be adjacent. These were introduced in full generality by \cite{MR1758852},
although special cases had been considered earlier.
\begin{example}
The permutation $\pi = 241635$ has two occurrences of the pattern $\vinb{1}\bl\vin{2\bl3}$,
given by the subsequences $135$ and $235$. These are also occurrences of the pattern
$\vinb{1}\bl\vinr{2\bl3}$ since they appear at the end of the permutation.
It avoids $\vin{1\bl2\bl3}$.
\end{example}

A \emph{bivincular} pattern is a further generalization where requirements on
the values that must be in the occurrence are allowed. Here we use two-line notation
and bars on adjacent letters in the top line mean that the corresponding letters in the
occurrence must have adjacent values. These were introduced by \cite{BCDK08}.

\begin{example}
The permutation $\pi = 241635$ has one occurrence of the pattern
$\biv{\nivs{1\bl2}\bl\nivbs{3}}{\vinbs{1\bl2\bl3} }$,
given by the subsequence $235$. This is also an occurrence of
$\biv{\nivs{1\bl2}\bl\nivrs{3} }{\vinbs{1}\bl\vinrs{2\bl3} }$, with the hook in
the top line signifying that the largest letter in the permutation is in the occurrence.
It has one occurrence
of $\biv{\nivbs{1}\bl\nivs{2\bl3} }{\vinbs{1\bl2\bl3} }$ given by the subsequence
$245$, which is also an occurrence of
$\biv{\nivbs{1}\bl\nivrs{2\bl3} }{\vinls{1\bl2}\bl\vinbs{3} }$.
It avoids $\biv{\nivs{1\bl2\bl3} }{\vinbs{1\bl2\bl3} }$.
\end{example}

We will also use the notation of \cite{BCDK08} to write bivincular patterns:
A bivincular pattern consists of a triple $(p,X,Y)$ where $p$ is a permutation
in $\symS_k$ and $X,Y$ are subsets of $\dbrac{0,k}$. An occurrence of this bivincular
pattern in a permutation $\pi = \pi_1\dotsm\pi_n$ in $\symS_n$ is a subsequence
$\pi_{i_1}\dotsm\pi_{i_k}$ such that the letters in the subsequence are in
the same relative order as the letters of $p$ and
\begin{itemize}

 \item for all $x$ in $X$, $i_{x+1} = i_x + 1$; and
 
 \item for all $y$ in $Y$, $j_{y+1} = j_y + 1$, where
 $\{ \pi_{i_1}, \dotsc, \pi_{i_k} \} = \{ j_1, \dotsc, j_k \}$ and
 $j_1 < j_2 < \dotsm < j_k$.
 
\end{itemize}
By convention we put $i_0 = 0 = j_0$ and $i_{k+1} = n+1 = j_{k+1}$.

The bivincular patterns behave well with respect to the operations 
reverse, complement and inverse: Given a bivincular pattern $(p,X,Y)$
we define
\begin{align*}
	(p,X,Y)^\text{r} &= (p^\text{r},k-X,Y),\\
	(p,X,Y)^\text{c} &= (p^\text{c},X,k-Y),\\
	(p,X,Y)^\text{i} &= (p^\text{i},Y,X),
\end{align*}
where $p^\text{r}$ is the usual reverse of the permutation of $p$,
$p^\text{c}$ is the usual complement of the permutation of $p$, and
$p^\text{i}$ is the usual inverse of the permutation of $p$. Here
$k-M = \{ k-m \sep m \in M\}$.

We get a very simple but useful Lemma:
\begin{lemma}
 Let \textup{a} denote one of the operations above (or their compositions).
 Then a permutation $\pi$ avoids the bivincular pattern $\bfp$ if and only
 if the permutation $\pi^\textup{a}$ avoids the bivincular pattern
 $\bfp^\textup{a}$.\qed
\end{lemma}

The following proposition will allow us to reduce the number of patterns that
we need to look at.

\begin{proposition} \label{prop:basic}
Below $\rma$ is one of $\rmr, \rmc, \rmi$, or a composition of them.
\begin{enumerate}

\item \label{prop:basic-1}
Assume that for any permutations $\pi, \mu$,
\[
\pi^\rma \sim \mu^\rma \text{ implies } \pi \sim \mu.
\]
Then the maps $\rma \colon \eqAn{\bfp} \to \eqAn{\bfp^\rma}$,
$\rma \colon \eqMn{\bfp} \to \eqMn{\bfp^\rma}$
are injections.

\item Assume that for any permutations $\pi, \mu$,
\[
\pi \sim \mu \text{ implies } \pi^\rma \sim \mu^\rma.
\]
Then the maps $\rma^{-1} \colon \eqAn{\bfp^\rma} \to \eqAn{\bfp}$,
$\rma^{-1} \colon \eqMn{\bfp^\rma} \to \eqMn{\bfp}$
are injections.

\item \label{prop:basic-3}
Assume that for any permutations $\pi, \mu$,
\[
\pi \sim \mu \text{ if and only } \pi^\rma \sim \mu^\rma.
\]
Then the maps $\rma \colon \eqAn{\bfp} \to \eqAn{\bfp^\rma}$,
$\rma \colon \eqMn{\bfp} \to \eqMn{\bfp^\rma}$
are bijections.

\end{enumerate}
\end{proposition}

\begin{proof}
It suffices to prove (\ref{prop:basic-1}) for the avoiding classes,
since the other cases are similar.
We first prove that the image of the map $\rma$ is actually in $\eqAn{\bfp^\rma}$.
Take $\pi$ in $\eqAn{\bfp}$, so $\pi$ and all equivalent permutations
avoid the pattern $\bfp$. We must show that $\pi^\rma$ and all equivalent
permutations avoid $\bfp^\rma$. But if $\rho$ is equivalent to $\pi^\rma$
and contains $\bfp^\rma$ then $\rho^{\rma^{-1}}$ is equivalent to
$(\pi^\rma)^{\rma^{-1}} = \pi$ and contains $(\bfp^\rma)^{\rma^{-1}} = \bfp$, which is a
contradiction.
Since the map $\rma$ is a bijection (on $\symS_n$) it follows that it's restriction
to $\eqAn{\bfp}$ is an injection.
\end{proof}

Often part (\ref{prop:basic-3}) will be implied by using:

\begin{lemma}
Assume that for any permutation $\pi$ we have $\pi \sim \pi^\rma$.
Then $\pi \sim \mu$ if and only $\pi^\rma \sim \mu^\rma$ and
the maps $\rma \colon \eqAn{\bfp} \to \eqAn{\bfp^\rma}$,
$\rma \colon \eqMn{\bfp} \to \eqMn{\bfp^\rma}$
are bijections.
\end{lemma}

The proposition above is mostly useful when one is looking for patterns
$\bfp$ that give interesting sets $\eqAn{\bfp}$. If for example we were
interested in toric equivalence and decide to look at all bivincular patterns
of length $3$, without referring the proposition, we would have to go through the entire list of $1536$ bivincular
patterns. By Theorem \ref{thm:toricsymmetry} we can reduce by all the basic
symmetries, which leaves us with only $212$ patterns to look at, making the
problem much more tractable.

For the rest of the section we assume that the equivalence relation has been extended to patterns.

\begin{lemma}
\begin{enumerate}

\item \label{prop:sub-stab1}
Let $\bfp$ be a pattern and assume that for any $\bfq \sim \bfp$ we have
\[
\eqAn{\bfp} \subseteq \eqAn{\bfq}.
\]
Then $\eqAn{\bfp} \subseteq \Aeqn{\bfp}$.

\item Let $\bfp$ be a pattern and assume that for any $\bfq \sim \bfp$ we have
\[
\eqMn{\bfp} \subseteq \eqMn{\bfq}.
\]
Then $\eqMn{\bfp} \subseteq \Meqn{\bfp}$.

\end{enumerate}

\end{lemma}

\begin{proof}
This is obvious.
\end{proof}

\begin{proposition} \label{prop:stable-patterns}
Assume that for any permutations $\pi, \mu$,
\[
\pi \sim \mu \text{ if and only } \pi^\rma \sim \mu^\rma,
\]
and that the same is still true if $\pi, \mu$ are replaced with patterns.
Then the pattern $\bfp$ is stable if and only if $\bfp^\rma$ is stable.
\end{proposition}

\begin{proof}
We assume $\bfp$ is stable. The statement follows from following the composition
below from left to right.
\[
\eqAn{\bfp^\rma} \to \eqAn{\bfp} = \Aeqn{\bfp} \to \widetilde{\symA}_n(\widetilde{\bfp}^\rma)
= \Aeqn{\bfp^\rma}.\qedhere
\]
\end{proof}

Whenever we reference integer sequences of the form Axxxxxx we are always
referring to the Online Encyclopedia of Integer Sequences, \cite{S10}.

\section{Conjugacy} \label{sec:Conjugacy}

Here we will regard two permutations as equivalent
if they have the same cycle type, or equivalently, $\pi \sim \mu$ if
there exists a permutation $\sigma$ such that $\pi = \sigma \mu \sigma^\rmi$.
We start by examining which symmetries behave nicely with respect to this
equivalence relation.

\subsection{Symmetry}

Since $\pi$ and $\pi^\rmi$ have the same cycle type (write
$\pi$ as a product of disjoint cycles, then obtain $\pi^\rmi$ by reversing each cycle),
we see that conjugacy is compatible with inverses.

It is also easy to see that $\pi$ and $\pi^{\rmr\rmc}$ have the same cycle type.
This can be seen by representing permutations as matrices.

\begin{example}
Let $\pi = 948167523$. Then the \emph{permutation matrix} of $\pi$ is
\[
M_\pi = \left(\begin{array}{ccccccccc}0 & 0 & 0 & 0 & 0 & 0 & 0 & 0 & 1 \\0 & 0 & 0 & 1 & 0 & 0 & 0 & 0 & 0 \\0 & 0 & 0 & 0 & 0 & 0 & 0 & 1 & 0 \\1 & 0 & 0 & 0 & 0 & 0 & 0 & 0 & 0 \\0 & 0 & 0 & 0 & 0 & 1 & 0 & 0 & 0 \\0 & 0 & 0 & 0 & 0 & 0 & 1 & 0 & 0 \\0 & 0 & 0 & 0 & 1 & 0 & 0 & 0 & 0 \\0 & 1 & 0 & 0 & 0 & 0 & 0 & 0 & 0 \\0 & 0 & 1 & 0 & 0 & 0 & 0 & 0 & 0\end{array}\right).
\]
To write $\pi$ as a product of disjoint cycles we start in the first row and
use the position of the $1$ in that row to tell us which row to visit next. Stop when
we loop around. This gives us
\[
\pi = (193824)(567).
\]
When we apply reverse and then complement it is like reading the matrix from the bottom up,
from right to left. We see that we will get the same cycle structure, with $i$ replaced
by $n+1-i$.
\end{example}


By Proposition \ref{prop:basic} we get the following theorem.

\begin{theorem} \label{thm:cycletypesymmetry}
Let $\rma$ be a composition of elements from the set
$\{\rmi, \rmr\rmc\}$. Then the map
$\rma: \eqAn{\bfp} \to \eqAn{\bfp^\rma}$ is a bijection for all $n$.
\end{theorem}


\subsection{Number of classes}

The number of equivalence classes in $\symS_n$ is
\[
1, 2, 3, 5, 7, 11, 15, 22, 30\dotsc \qquad n = 1,2,3 \dotsc.
\]
This is just the number of conjugacy classes in $\symS_n$ or equivalently
the number of partitions of $n$, A000041.

\subsection{Pattern avoidance}

We start with a very easy result.

\begin{proposition}
\begin{enumerate}

 \item The pattern $\bfp = 12$ gives
 \[
 \eqAn{\bfp} = \tom \qquad (n \geq 3).
 \]
 %
  

 
 \item The pattern $\bfp  = \biv{ \kern 2pt \nivrs{1\bl2} }{ \vinls{1}\bl\vinbs{2} }$
gives
 \[
 \eqAn{\bfp} = \begin{cases}
      \id^\rmr & \text{ if $n = 2$}, \\
      \id      & \text{otherwise}.
\end{cases}
 \]
 
\end{enumerate}
\end{proposition}

\begin{proof}
\text{}
\begin{enumerate}

\item  Take any permutation that avoids this pattern. It must be
 $n (n-1) \dotsm 3 2 1$. This pattern has cycle type $(2,2, \dotsc, 2,1)$
 if $n$ is odd and $(2,2, \dotsc, 2,2)$ if $n$ is even. The permutations
 $2 1 4 3 \dotsm (n-1) (n-2) n$ (if $n$ is odd) or $2 1 4 3 \dotsm n (n-1)$
 (if $n$ is even) have the same cycle type, but contain the pattern. Therefore
 the count is always zero. This argument works for $n \geq 3$ but not in $\symS_1$ and
 $\symS_2$.

\item
 I claim that $\eqAn{\bfp}$ consists of the identity permutation, except
 when $n=2$ then it is replaced by $21$. It is clear that these permutations
 are in the set, we just need to show that there are no others. To do this
 we must show that we can for any cycle type different from $(1,1\dotsc,1)$
 (for the identity) we can create a permutation $\pi = (n-1) \dotsm n \dotsm$ with
 that cycle type. So take a permutation $\pi'$ that contains a cycle of
 length $\geq 2$. Let $(a b \dotsm)$ be the start of that cycle. Now just
 conjugate with the $(1 a)(b (n-1))$ and we have preserved the cycle type and
 put $n-1$ and the beginning in one-line notation.\qedhere

\end{enumerate}
\end{proof}

The next result is a special case of Proposition \ref{prop:CyclT:avoid-k-cycle}).

\begin{proposition} \label{prop:CyclT:derangements}
The pattern $\bfp = \biv{\nivls{1}}{\vinls{1}} = (1, \{0\}, \{0\})$ gives
\begin{align*}
\eqAn{\bfp} &= \text{derangements in $\symS_n$} \\
|\eqAn{\bfp}| &= 0, 1, 2, 9, 44, 265, 1854, 14833, 133496, \dotsc, \qquad n = 1,2,3,\dotsc
\end{align*}
which is sequence A000166. This sequence has the generating function $\frac{e^{-x}}{1-x}$.
\end{proposition}

\begin{proof}
 A permutation has no fixed points
 if and only its cycle type contains no $1$. Call the set of derangements in
 $\symS_n$, $D_n$.
 
 ($D_n \subseteq \symS_n$): Take $\pi$ not in $D_n$, then some permutation
 with the same cycle type as $\pi$ contains the pattern. This implies that
 $\pi$ has a $1$ in its cycle type so $\pi$ has a fixed point. Therefore $\pi$
 is not in $D_n$.
 
 ($\symS_n \subseteq D_n$): Take $\pi$ not in $\symS_n$, so $\pi$ has a fixed
 point and therefore its cycle type contains a $1$. Let $a$ be the fixed point.
 Swap $1$ and $a$ in the one-line notation for $\pi$ and we turn $1$ into a
 fixed point, while preserving the cycle type.
\end{proof}

By \cite[Exercise 7, Chapter 2]{MR847717} we have that the number of permutations
of $\dbrac{n}$ that have no $k$-cycle is given by the generating function
\[
\frac{e^{-x^k/k}}{1-x}.
\]
The proposition above is the special case $k = 1$ which means the permutations
have no $1$-cycles, i.e., fixed points. It turns out that the pattern that
appears in the proposition is part of a family of patterns:

\begin{proposition} \label{prop:CyclT:avoid-k-cycle}
Let $k \geq 1$
The pattern
\begin{align*}
\bfp =& \biv{ \nivl{1\bl2\bl3\bl\dotsm\bl k} }{ \vinl{2\bl3\bl\dotsm\bl k\bl1} }
\end{align*}
gives
\[
\eqAn{\bfp} = \text{permutations in of $\dbrac{n}$ that do not contain a $k$-cycle},
\]
the count has generating function $\frac{e^{-x^k/k}}{1-x}$. For $k = 1,\dotsc,7$ we get
sequences A000166, A000266, A000090, A000138, A060725, A060726, A060727.
\end{proposition}

\begin{proof}
This is a slight generalization of the proof above and omitted.
\end{proof}

The next two propositions have similar proofs to Proposition \ref{prop:CyclT:derangements},
so we omit the proofs.

\begin{proposition} \label{prop:CyclT:involutions}
The pattern $\bfp = \biv{\nivls{1\bl2\bl3}}{\vinls{2\bl3}\bl\vinbs{1}}$ 
gives
\begin{align*}
\eqAn{\bfp} &= \text{involutions in $\symS_n$} \\
|\eqAn{\bfp}| &= 1, 2, 4, 10, 26, 76, 232, 764, 2620, \dotsc, \qquad n = 1,2,3,\dotsc
\end{align*}
which is A000085. This sequence has the generating function $\exp(x+x^2/2)$.
\end{proposition}

It turns out that this pattern is also part of a family of patterns:

\begin{proposition}
Let $k \geq 1$
The pattern
\[
\bfp = \biv{ \nivl{1\bl2\bl3\bl\dotsm\bl k} }
           { \vinl{2\bl3\bl\dotsm\bl k}\bl\vinb{1} }
\]
gives
\[
\eqAn{\bfp} = \text{permutations in of $\symS_n$ only containing cycles of length $< k$},
\]
the count has the generating function
\[
\exp{\left( x + \frac{x^2}{2} + \frac{x^3}{3} + \dotsb +
	 \frac{x^{k-1}}{k-1} \right)}.
\]
For $k = 1,\dotsc,7$ we get
sequences A000004 (The zero sequence), A000012 (The all $1$'s sequence), A000085, A057693, A070945,
A070946, A070947.
\end{proposition}

%
%

\begin{proposition}
The pattern $\bfp = \biv{ \nivls{1}\bl\nivrs{2} }{ \vins{2\bl1} }$ 
gives
 \begin{align*}
 \eqAn{\bfp} &= \text{ $\id$ and transpositions in $\symS_n$} \\
 |\eqAn{\bfp}| &= 1, 1, 4, 7, 11, 16, 22, 29, 37, \dotsc, \qquad n = 1,2,3,\dotsc
 \end{align*}
 If we count from $n = 3$ we get $T(n+1)+1$ where $T(n) = \frac{n(n+1)}{2}$ is
 A000217, the triangular numbers.
 \end{proposition}

\begin{proof}
I claim that $\eqAn{\bfp}$ consists of
permutations with cycle type $(1,1,\dots,1)$ or $(2,1,\dotsc,1)$. Call the set
with those kinds of cycle type $Y_n$. 

($Y_n \subseteq \eqAn{\bfp}$):
$\eqAn{\bfp}$ clearly contains the identity permutation with the former cycle type.
So consider any permutation $\pi$ with the latter cycle type. This permutation
will consist of a transposition and then everything else is fixed. A permutation
of this type clearly avoids $\bfp$.

($\eqAn{q} \subseteq Y_n$):
Take $\pi$ that is not in $Y_n$. Then we must consider two cases
\begin{enumerate}
\item
The cycle type of $\pi$ has a cycle of length $\geq 3$: Let this cycle
start with $(a b c \dotsm)$. Then just conjugate $\pi$ with
$(c 1)(b n)((n-1) a)$, which preserves the cycle type but introduces
an occurrence of the pattern.

\item
The cycle type of $\pi$ has two cycles, both of length $\geq 2$.
We can assume these cycles have length exactly $2$, since if they
were larger we can just apply the argument above to either one of them.
So let these cycles be $(a b)$ and $(c d)$. Then just conjugate $\pi$
with $(a 1)(b (n-1))(c (n-2))(d n)$. \qedhere
\end{enumerate}
\end{proof}

\begin{proposition} \label{prop:centralpolyg}
The pattern $\bfp = \biv{\nivbs{1\bl2}\bl\nivrs{3}}{\vinbs{2\bl3\bl1}}$ gives
\begin{align*}
\eqAn{\bfp} &= \text{ permutations moving at most two letters} \\
|\eqAn{\bfp}| &= 1,2,4,7,11,16,22,29,37, \dotsc, \qquad n = 1,2,3,\dotsc
\end{align*}
This is A000124: Central polygonal numbers, $n(n-1)/2 +1$.
\end{proposition}

\begin{proof}
Let $X_n$ be the set on the right. It is clear that $X_n \subseteq \eqAn{\bfp}$.
The other implication is also simple: given a cycle of length $3$ or more
we can conjugate to produce a permutation in the same class that looks like
$2n \dotsm 1 \dotsm$. Finally, given two $2$-cycles we can conjugate to
produce $3n1 \dotsm 2$.
\end{proof}

\begin{openproblem}
The permutations above are equinumerous (but not equal to) permutations
avoiding $132$- and $321$-avoiding permutations. It would be interesting
to produce a bijection.
\end{openproblem}

The next three propositions have proofs that are similar to the proof of
Proposition \ref{prop:centralpolyg}, so we omit their proofs.

\begin{proposition}
The pattern $\bfp = \biv{ \nivls{1\bl2}\bl\nivrs{3} }{ \vinbs{1\bl3}\bl\vinrs{2} }$ gives
\begin{align*}
\eqAn{\bfp} &= \text{ fixed point free involutions} \\
|\eqAn{\bfp}| &= 1,2,3,4,1,16,1,106,1 \dotsc, \qquad n = 1,2,3,\dotsc \\
 &= \begin{cases}
    1 & \text{ if $n$ is odd}, \\
    (n-1)!!+1  & \text{ if $n$ is even}.
\end{cases}
\end{align*}
The even subsequence is A001147: Double factorial numbers.
\end{proposition}

\begin{proposition}
The pattern $\bfp = \biv{\nivrs{1\bl2\bl3}}{\vinbs{1\bl3}\bl\vinrs{2}}$ gives
\begin{align*}
\eqAn{\bfp} &= \text{ the identity and $3$-cycles} \\
|\eqAn{\bfp}| &= 1,2, 3,9, 21, 41, 71,113,169 \dotsc, \qquad n = 1,2,3,\dotsc \\
 &= 1+2\binom{n}{3}.
\end{align*}
From $n=3$ this is A007290 plus $1$.
\end{proposition}

\begin{proposition}
The pattern $\bfp = \biv{\nivrs{1\bl2\bl3}}{\vinls{2}\bl\vinbs{3\bl1}}$ gives
\begin{align*}
\eqAn{\bfp} &= \text{ the identity, $2$-cycles and $3$-cycles} \\
|\eqAn{\bfp}| &= 1,2, 4, 15, 31, 56, 92, 141, 205, \dotsc, \qquad n = 1,2,3,\dotsc 
\end{align*}
From $n=4$ this is A000330 plus $1$, where A000330 are the square pyramidal numbers,
$\frac{n(n+1)(2n+1)}{6}$.
\end{proposition}

\begin{openproblem}
The number of permutations avoiding $\vin{13}\vinb{2}$ that contain the pattern $\vin{32}\vinb{1}$
exactly once is counted by the square pyramidal numbers. Find a bijection to the permutations
above (leaving out the identity).
\end{openproblem}

%

\section{Order} \label{sec:Order}

Here we will regard two permutations as equivalent if
they have the same order. Clearly the equivalence classes here will be unions
of conjugacy classes and therefore the results in this section will rely on results from
the previous one.

\subsection{Symmetry}

The following is a direct consequence of Theorem \ref{thm:cycletypesymmetry}:

\begin{theorem} \label{thm:ordersymmetry}
Let $\rma$ be a composition of elements from the set
$\{\rmi, \rmr\rmc\}$. Then
$\rma: \eqAn{\bfp} \to \eqAn{\bfp^\rma}$ is a bijection for all $n$.
\end{theorem}
%

\subsection{Number of classes}
The number of equivalence classes in $\symS_n$ is
\[
1,2,3,4,6,6,9,11,14 \qquad n = 1, \dotsc 9.
\]
Of course this is just A009490: Number of distinct orders of permutations of $n$ objects.

\subsection{Pattern avoidance}

\begin{proposition} \label{prop:ordernotfrom1cycl}
The pattern $\bfp = \biv{\nivl{1}}{\vinl{1}}$ gives
\begin{align*}
\eqAn{\bfp} = &\text{ permutations with order that can not be obtained from} \\
              &\text{ permutations with a fixed point ($1$-cycle).} \\
            = &\text{ permutations with order that can not be written $\prod a_i$} \\
              &\text{ with $1 + \sum a_i = n$.} \\
|\eqAn{\bfp}| &= 0, 1, 2, 6, 44, 0, 1644, 7728, 84384, \dotsc, \qquad n = 1,2,3,\dotsc
\end{align*}
\end{proposition}

\begin{proof}
This follows directly from Proposition \ref{prop:CyclT:derangements}.
\end{proof}

This result can also be generalized in the same way as Proposition \ref{prop:CyclT:derangements}
above:

\begin{proposition}
Let $k \geq 1$
The pattern $\bfp = \biv{ \nivl{1\bl2\bl3\bl\dotsm\bl k} }{ \vinl{2\bl3\bl\dotsm\bl k\bl1} }$
gives
\begin{align*}
\eqAn{\bfp} = &\text{ permutations with order that can not be obtained from}\\
              &\text{ permutations with a $k$-cycle,} \\
            = &\text{ permutations with order that can not be written
               $k \cdot m$ with $k \nmid m$.} \\
\end{align*}
\end{proposition}

\begin{proof}
This follows directly from Proposition \ref{prop:CyclT:avoid-k-cycle}.
\end{proof}

\begin{proposition}
The pattern $\bfp = \biv{\nivl{1\bl2\bl3}}{\vinl{2\bl3}\bl\vinb{1}}$ gives
\begin{align*}
\eqAn{\bfp} &= \text{involutions in $\dbrac{n}$} \\
|\eqAn{\bfp}| &= 1, 2, 4, 10, 26, 76, 232, 764, 2620, \dotsc, \qquad n = 1,2,3,\dotsc
\end{align*}
which is A000085 in OEIS. This sequence has the generating function $\exp(x+x^2/2)$.
\end{proposition}

\begin{proof}
This follows directly from Proposition \ref{prop:CyclT:involutions}.
\end{proof}

This generalizes in a similar way we saw with Proposition \ref{prop:ordernotfrom1cycl}
above, but we omit it.

%

\section{Knuth-equivalence} \label{sec:Knuth}

\begin{definition} \label{def:Knuth-eq}
Let $A$ be an alphabet with an ordering.
\begin{enumerate}
\item An \emph{elementary Knuth-transformation} on a word with letters from
$A$ applies one of the transformations below, or their
inverses, to three consecutive letters in the word.
\begin{itemize}
\item[K1] $yzx \mapsto yxz$ if $x < y \leq z$,
\item[K2] $xzy \mapsto zxy$ if $x \leq y < z$.
\end{itemize}
\item Two words $w$ and $w'$ are said to be \emph{Knuth-equivalent} if they can be changed into
each other by a sequence of elementary Knuth-transformations. We write $w \equiv w'$ if this is the case.
\end{enumerate}
\end{definition}

\noindent
In this paper we will only consider the alphabet $A = \NN = \{1,2,3, \dots \}$
so K1 means we can interchange $zx$ if the next letter to the left fits between them;
and K2 means that we can interchange $xz$
if the next letter to the right fits between them. E.g., $24135 \equiv 21435 \equiv 21453$.

As is shown in \cite{MR1464693}, two permutations $\pi$ and $\mu$ are Knuth-equivlent
if and only if they have the same insertion tableau. Below we will denote the
insertion tableau of $\pi$ with $P(\pi)$ and the recording tableau with $Q(\pi)$.

\subsection{Symmetry}

It is easy to see from Definition \ref{def:Knuth-eq} that if $\rma = \rmr$ or $\rma = \rmc$
then $\pi \sim \mu$ if and only if $\pi^\rma \sim \mu^\rma$. Then Proposition \ref{prop:basic}
implies the following theorem.

\begin{theorem} \label{thm:knuthymmetry}
Let $\rma$ be a composition of elements from the set
$\{\rmr, \rmc\}$. Then
$\rma: \eqAn{\bfp} \to \eqAn{\bfp^\rma}$ is a bijection for all $n$.
\end{theorem}


\subsection{Number of classes}

The number of equivalence classes in $\symS_n$ is
\[
1, 2, 4, 10, 26, 76, 232, 764, 2620\dotsc \qquad n = 1, \dotsc.
\]
This is A000085: number of Young tableaux with $n$ cells.

\subsection{Pattern avoidance}

Here we extend the relation to patterns by saying that two patterns
$(p,X,Y)$, $(p',X',Y')$ are equivalent if and only if $p$ and $p'$
are equivalent, and $X = X'$, $Y = Y'$.

\begin{proposition} \label{knuthprop:stabl}
For $k \geq 1$ we have that the pattern $\bfp = 12 \dotsm k$ satisfies
\begin{align*}
\eqAn{\bfp} &= \symA_n(\bfp).
\end{align*}
The members of these sets are permutations whose longest increasing subsequence is of length less than
$k$. For $n = 1,\dotsc 7$ we get the sequences A000004 (zero sequence),
A000012 (the $1$'s sequence), A000108 (the Catalan numbers), A005802 (number of vexillary
permutations), A047889, A047890, A052399. These sequences are studied in \cite{MR1788170}.
\end{proposition}

\begin{proof}
We always have $\eqAn{\bfp} \subseteq \symA_n(\bfp)$ so take a permutation
$\pi \in \symA_n(\bfp)$. This implies that the first row of the tableau
$P(\pi)$ has length $\leq k$. The same is then true for every equivalent permutation
which implies that every equivalent permutation also avoids the pattern.
\end{proof}

\begin{corollary} \label{cor:knuth-first-stable}
The pattern $\bfp = 12 \dotsm k$ is stable for $k \geq 1$.
\end{corollary}

\begin{proof}
Since $\tilde{\bfp} = \{\bfp\}$ for these patterns the result follows from
Proposition \ref{knuthprop:stabl}.
\end{proof}

\begin{proposition}
For $k \geq 1$ we have that
\[
\eqAn{\vin{12 \dotsm k}} = \eqAn{12 \dotsm k}.
\]
This implies that the patterns $\vin{12 \dotsm k}$ are not stable,
unless $k = 1,2$.
\end{proposition}

\begin{proof}
We obviously have $\eqAn{\vin{12 \dotsm k}} \supseteq \eqAn{12 \dotsm k}$,
so assume $\pi$ is a permutation that is not in the set on the right.
This means that the first row of $P(\pi)$ is of length $\geq k$. To construct
an equivalent permutation that contains the consecutive pattern, apply the
inverse RSK-correspondance to the tableaux-pair $(P(\pi),Q)$ where $Q$ is
filled in trivially (by reading from left to right, the top row first).
\end{proof}

\begin{proposition}
For $k \geq 1$ we have that the pattern $\bfp = \biv{\nivls{1\bl2\dotsm k}}{\vinbs{1\bl2\dotsm k}}$
gives
\begin{align*}
\eqAn{\bfp} &= \text{ permutations whose insertion tableu does not start with $\bfp$}
\end{align*}
these are permutations whose longest increasing subsequence is of length less than
$k$.
\end{proposition}

\begin{proof}
Similar to the proof above.
\end{proof}

\begin{openproblem}
The count for the permutations above seems to be given by $n! - \frac{n!}{k!}$ for $k \geq n$.
Prove this.
\end{openproblem}

\begin{proposition} \label{knuthprop:stabl2}
For $k \geq 1$ we have that the pattern $\bfp = \biv{\nivs{1\bl2\dotsm k}}{\vinbs{1\bl2\dotsm k}}$ satisfies
\begin{align*}
\eqAn{\bfp} &= \symA_n(\bfp).
\end{align*}
For $n = 1, \dotsc, 7$ we get the sequences A000004 (zero sequence),
A000012 (the $1$'s sequence), A049774, A117158, A177523, A177533, A177553.
These sequences have been studied in by \cite{Hardin}
\end{proposition}

\begin{proof}
We obviously have $\eqAn{\bfp} \subseteq \symA_n(\bfp)$,
so assume $\pi$ is a permutation that is not in the set on the left.
Let $\pi'$ be an equivalent permutation that contains the pattern.
This means that there is a row of $P(\pi')$ that contains the
sequence $\ell, \ell+1, \dots, \ell + k-1$ in adjacent boxes.
Since $P(\pi) = P(\pi')$ we see that $\pi$ must also contain the pattern.
\end{proof}

\begin{corollary} \label{cor:knuth-second-stable}
The pattern $\bfp = \biv{\nivs{1\bl2\dotsm k}}{\vinbs{1\bl2\dotsm k}}$ 
is stable for $k \geq 1$.
\end{corollary}

\begin{proof}
Since $\tilde{\bfp} = \{\bfp\}$ for these patterns the result follows from
Proposition \ref{knuthprop:stabl2}.
\end{proof}

\begin{proposition} \label{prop:orig}
The pattern $\bfp = 231$ gives
\begin{align*}
\eqAn{\bfp} &= \text{permutations with hook-shaped insertion tableaux, filled in trivially} \\
|\eqAn{\bfp}| &= 1, 2, 4, 8, 16, 32, 64, 128, 256 \dotsc, \qquad
n = 1,2,3,\dotsc
\end{align*}
which is 
A000079: $2^{n-1}$.
\end{proposition}

Before we start proving this we make a technical definition for convenience:
\begin{definition}
Let $\bfp$ be a pattern. Given an occurrence of $\bfp$ in a permutation $\pi$
we define the \emph{area} of the occurrence as the distance from the first letter
in the occurrence to the last.
A \emph{minimal occurrence} of a pattern $\bfp$ in a permutation $\pi$ is an
occurrence of $\bfp$ with the least possible area.
\end{definition}

\begin{proof}
We begin by showing that the permutations in the set on the left are exactly
the permutations that avoid $231$ and $213$. Then \cite[Proposition 5.3]{MR2061380} implies
the result.

Now, let $\pi$ be a permutation belonging to the set on the left. We need to show that
$\pi$ avoids the pattern $213$. First note that $\pi$ must avoid the pattern $\vin{2\bl1\bl3}$
since an occurrence of $\vin{2\bl1\bl3}$ can be changed into an occurrence of $\vin{2\bl3\bl1}$ by an elementary
transformation. It therefore suffices to show that an occurrence of $213$ can be
changed into an occurrence of $\vin{2\bl1\bl3}$ by elementary transformations. Let
\[
\pi = \dotsm \check{x} \dotsm \check{y} \dotsm \check{z} \dotsm
\]
be a minimal occurrence of $213$; so $y < x < z$, every letter $w$ between
$x$ and $y$ satisfies $w > z$ and every letter $w$ between $y$ and $z$ satisfies
$y < w < x$. Now any $w$ between $x$ and $y$ with $w > z$ would give an occurrence
of $231$ so $x$ and $y$ are actually adjacent.
If $y$ and $z$ are adjacent we are done, otherwise let $w$ be the first letter to the left of
$y$.
\[
\pi = \dotsm \check{x} \check{y} w \dotsm \check{z} \dotsm.
\]
We can swap $x$ and $y$ by an elementary transformation and we still have an
occurrence of $213$
\[
\pi \equiv \dotsm y \check{x} \check{w} \dotsm \check{z} \dotsm,
\]
and the number of letters between $w$ and $z$ is now one less than the number of
letters that were between $y$ and $z$ in $\pi$. We can now perform elementary
transformations until we have gotten rid of all the letters between $x$ and $z$, but
one, and this will be an occurrence of $\vin{2\bl1\bl3}$.

Now let $\pi$ be a permutation belonging to the union on the right. We need to show that
any permutation Knuth-equivalent to $\pi$ also avoids $231$. It suffices to consider
a permutation $\rho$ that differs from $\pi$ by one elementary transformation. We must
look at two cases:
\begin{enumerate}
\item
\[
\rho = \dotsm xyz \dotsm \equiv \dotsm yxz \dotsm = \pi,
\]
with $x < z < y$. If $\rho$ has an occurrence of $231$ then it would either
have to be
\[
\rho = \dotsm \check{x} y \check{z} \dotsm \check{w} \dotsm,
\]
or
\[
\rho = \dotsm \check{x} \check{y} z \dotsm \check{w} \dotsm.
\]
The first scenario would immediately give an occurrence of $231$ in $\pi$,
which is a contradiction, and the second scenario we would have also get an occurrence
of $231$:
\[
\pi = \dotsm y \check{x} \check{z} \dotsm \check{w} \dotsm.
\]

\item
\[
\rho = \dotsm zxy \dotsm \equiv \dotsm zyx \dotsm = \pi,
\]
This case is similar to the previous one and left to the reader. \qedhere
\end{enumerate}
\end{proof}

\begin{corollary} \label{cor:knuth-S3stable}
Every classical pattern in $\symS_3$ is stable.
\end{corollary}

\begin{proof}
The case $p = 123$ follows from Proposition \ref{cor:knuth-first-stable}
and the proof of the Proposition above implies the case $p = 231$.
By taking reverse and complement and using Proposition \ref{prop:stable-patterns}
we get the rest of $\symS_3$.
\end{proof}


We will need an obvious lemma regarding hook-shaped tableaux in the
proof of the next proposition:

\begin{lemma} \label{lem:hook-shaped}
If the tableau $P(\pi)$ is hook-shaped then in each step of it's construction when an element
is bumped from row $i$, it is smaller than the current element in row
$i+1$.
\end{lemma}

\begin{proposition} \label{prop:Graphs}
The pattern $\bfp = \biv{\nivls{1}\bl\nivbs{2\bl3}}{\vinbs{2\bl3\bl1}}$ gives
\begin{align*}
\eqAn{\bfp} = &\text{ permutations with hook-shaped insertion tableaux} \\
              &\text{ with $2$ in the first row, as well as $\id^\rmr$}, \\
|\eqAn{\bfp}| & = 1, 2, 4, 11, 36, 127, 463, 1717, 6436 \dotsc, \qquad
n = 1,2,3,\dotsc 9
\end{align*}
which is (from $n = 3$)
A112849: Number of congruence classes (epimorphisms/vertex partitionings induced by
graph endomorphisms) of undirected cycles of even length: $|C(C_2n)|$.
See \cite{MR2548551}.
\end{proposition}

\begin{proof}
Assume $\pi$ is not in the set on the left, and let $\pi'$ be an
equivalent permutation that matches the pattern. We can write
\[
\pi' = \dotsm j \dotsm k \dotsm 1 \dotsm, \qquad j \leq k.
\]
We can assume that $j$ and $k$ are adjacent. If $2$ appears before
$1$ in $\pi'$ then it will be bumped from the first row (by $1$),
and then $\pi'$ and $\pi$ would not be in the set on the right.
So assume that $2$ appears after $1$. Then the tableaux $P(\pi') = P(\pi)$
will not be hook-shaped by Lemma \ref{lem:hook-shaped}.

Now assume $\pi$ is not in the set on the right. We need to look at two cases:
\begin{enumerate}

\item $2$ is not in the first row of $P(\pi)$. This means that $1$ appears
after $2$ in $\pi$. If $1$ and $2$ are not adjacent in $\pi$ then we get
an occurrence of the pattern so we can assume that they are adjacent.
If the block $21$ is not at the end of the permutation then an elementary
Knuth-swap can be made to produce the pattern, so we can assume that
$\pi = \dotsm 21$. It is now clear that $\pi$ avoids the pattern if and only
if $\pi = \id^\rmr$.

\item Now assume that that $P(\pi)$ is not hook-shaped. The permutation
we get from
\[
P(\pi), \young(12\cdot\cdot\cdot\cdot,34\cdot\cdot,\cdot\cdot,\cdot,\cdot)
\]
by the RSK-correspondence will be equivalent to $\pi$. It is easy to see
that it will contain the pattern. \qedhere
\end{enumerate}
\end{proof}
\cite{MR2548551} give a formula for their sequence:
\[
1 + \frac{1}{2}\binom{2n-3}{n-2} + \frac{1}{2}\binom{2n-3}{n-1}.
\]
It is more natural to count the tableaux that appear with 
\[
1 + \sum_{i=2}^n \binom{n-2}{i-2} \cdot \binom{n-1}{i-1}
= 1 + \frac{1}{n-1}\sum_{i=2}^2 (i-1) \binom{n-1}{i-1}^2
= 1 + \frac{1}{2} \binom{2n-2}{n-1},
\]
which gives the same numbers as their formula.

\begin{openproblem} \label{openp:Graphs}
Find a bijection between the permutations above and the class enumerated in
\cite{MR2548551}.
\end{openproblem}

\subsection{Pattern matching}

Before moving to the next equivalence we state and prove a result
on equivalence classes that contain only permutations matching a certain
pattern.

\begin{proposition} \label{prop:knuth-matching}
The pattern $\bfp = \biv{\nivrs{1\bl2}}{\vinls{1}\bl\vinbs{2}}$ gives
\begin{align*}
\eqMn{\bfp} &= \text{permutations of the form $(n-1)\rho$ where $\rho$ avoids $123$} \\
|\eqMn{\bfp}| &= 1, 2, 5, 14, 42, 132, 429, 1430, \dotsc, \qquad n = 2,\dotsc
\end{align*}
which is A000108 in OEIS, the Catalan numbers, except this is shifted, so we get
$2$ permutations in $\symS_3$, for example.
\end{proposition}

\begin{proof}
The enumeration is obvious if we can prove the equality of the two sets.
Assume that $\pi$ is not in the set on the left, and that $\pi$ does
start with $n-1$. We can find a $\pi'$, equivalent to $\pi$, that does not
match the pattern, and we can assume that $\pi'$ is just one elementary swap
from $\pi$. We have
\begin{align*}
\pi &= (n-1)k\ell\dotsm n \dotsm \\
\pi' &= k(n-1)\ell \dotsm n \dotsm,
\end{align*}
where $k < \ell < n-1$ (we can not have $k,\ell = n$). Then $\pi$ contains $123$.

Now assume that $\pi$ is not in the set on the right and that $\pi$ starts
with $n-1$. Since $\pi$ contains the pattern $123$ we see that the first row
of $P(\pi)$ has at least three boxes. The equivalent permutation we get from
RSK of
\[
P(\pi), \young(12\cdot\cdot\cdot k,\cdot\cdot,\cdot\cdot,\cdot)
\]
does not start with $n-1$, so $\pi$ is not in the set on the left.
\end{proof}

It is well-known that the Catalan numbers enumerate tableaux of shape $(2,n)$
so that explains what tableaux appear above.

%

\section{Toric equivalence} \label{sec:Toric}
I learned about this equivalence relation from Anthony Labarre,
who responded to my question on \cite{MO}.

If $\lambda$ is a circular
permutation of $\dbrac{0,n}$ then $\lambda_\circ$ denotes the permutation in $\symS_n$ we get
by reading $\lambda$ from $0$. E.g., if $\lambda = 130254$ then $\lambda_\circ = 25413$.

Our definition of toric equivalence follows \cite{MR1861423}, but as noted
there, an equivalent class of objects
was studied by \cite{St07}.
Here the relation can be roughly viewed as declaring two permutations to be
equivalent if their
permutation matrices become equal when they are wrapped around a torus. More precisely, given a permutation
$\pi$ of $\symS_n$ we define $\pi^\circ$ as the circular permutation $0\pi$ of $\dbrac{0,n}$. Then
for any $m = 0, 1, \dotsc, n$ we define a new circular permutation
\[
\pi^\circ \oplus m = (0+m)(\pi_1 + m)(\pi_2 + m) \dotsm (\pi_n+m) \bmod{(n+1)}.
\]
(Here every letter is reduced modulo $n+1$). It is also convenient to define
\[
\pi \oplus m = (\pi^\circ \oplus m)_\circ.
\]
Then the \emph{toric class} of the original permutation
$\pi$ is defined as the set
\[
\pi^\circ_\circ = \{ \pi \oplus m \sep m = 0,1,\dotsc,n\}.
\]

\begin{example}
Let $\pi = 1243$. Then $\pi^\circ = 01243$ and
\begin{align*}
\pi^\circ \oplus 0 &= 01243,\quad \pi^\circ \oplus 1 = 12304,\quad \pi^\circ \oplus 2 = 23410, \\
\pi^\circ \oplus 3 &= 34021, \quad \pi^\circ \oplus 4 = 40132.
\end{align*}
The toric class of $\pi$ is $\pi^\circ_\circ = \{1243, 4123, 2341, 2143, 1324\}$.
\end{example}

\subsection{Symmetry}

\begin{lemma}
If $\lambda$ is a circular permutation of $\dbrac{0,n}$ and $\pi$, $\rho$
are two representations then $\pi^\rmi \equiv \rho^\rmi$.
\end{lemma}

Note that when we take the complement of a permutation of $\dbrac{0,n}$
(modulo $n+1$) then
$0$ and stays fixed, and it is the only letter that does so if $n$ is even;
if $n$ is odd then $(n+1)/2$ also stays fixed.

\begin{lemma}\label{lem:symmtor}
For any permutation $\pi$ of $\dbrac{1,n}$ we have:
\begin{enumerate}

\item \label{lem:symmtor-reverse}
Taking reverse commutes with the $\circ$-operator:
\[
(\pi^\rmr)^\circ \equiv (\pi^\circ)^\rmr, \textrm{ or equivalently }
((\pi^\circ)^\rmr)_\circ = \pi^\rmr.
\]
Also, for any circular permutation $\lambda$ of $\dbrac{0,n}$ we have
\[
\lambda^\rmr \oplus m \equiv (\lambda \oplus m)^\rmr.
\]

\item \label{lem:symmtor-complement}
Taking complement commutes with the $\circ$-operator:
\[
(\pi^\rmc)^\circ \equiv (\pi^\circ)^\rmc, \textrm{ or equivalently }
((\pi^\circ)^\rmc)_\circ = \pi^\rmc.
\]
Also, for any circular permutation $\lambda$ of $\dbrac{0,n}$ we have
\[
\lambda^\rmc \oplus m \equiv (\lambda \ominus m)^\rmc.
\]

\item \label{lem:symmtor-inverse}
Taking inverse commutes with the $\circ$-operator:
\[
(\pi^\rmi)^\circ \equiv (\pi^\circ)^\rmi, \textrm{ or equivalently }
((\pi^\circ)^\rmi)_\circ = \pi^\rmi.
\]
Also, for any circular permutation $\lambda$ of $\dbrac{0,n}$ we have
\[
\lambda^\rmi \oplus x \equiv (\lambda \oplus 1)^\rmi,
\]
where $x$ is the distance (counter-clock-wise) from $n$ to $0$ in $\lambda$.
More generally, we have
\[
\lambda^\rmi \oplus mx \equiv (\lambda \oplus m)^\rmi,
\]
where $x$ has the same value as above.

\end{enumerate}
\end{lemma}

\begin{proof}
Parts (\ref{lem:symmtor-reverse}) and (\ref{lem:symmtor-complement}) are left to
the reader, and we only prove part (\ref{lem:symmtor-inverse}). It suffices to
prove that part in the case $m=1$ as the other cases can be deduced by induction.
The first element after $0$ on the right-hand side is the location of $1$ (from $0$) in
$\lambda \oplus 1$. This equals the location of $0$ (from $n$) in $\lambda$.

The first element after $0$ on the left-hand side equals (the first element
after $n+1-x$ in $\lambda^\rmi$)$+x$. Now the first element after $n+1-x$ in $\lambda^\rmi$
records the location of the element in $\lambda$ that is one larger than the element
at location $n+1-x$. But the element at location $n+1-x$ is $n$, so the first element after
$n+1-x$ records the location of $0$, which is at location $0$. When we add $x$ to this
we get $x$ back. This argument generalizes to every location from $0$.
\end{proof}

\begin{example}
Let $\lambda = 04372156$. Here the distance from $n$ to $0$ is $x=5$.
Then $\lambda \oplus 1 = 15403267$ and $(\lambda \oplus 1)^\rmi = 63405217$.
This is the right-hand side of part (\ref{lem:symmtor-inverse}) above. To calculate
the left-hand side we first have $\lambda^\rmi = 05421673$ and then
$\lambda^\rmi \oplus 5 = 52176340$.
\end{example}

\noindent
Some of the equations in Lemma \ref{lem:symmtor} are also true when regarding the
permutations as usual permutations of $\dbrac{0,n}$, and not just as circular.

\begin{proposition}
Let $\rma$ be any composition of the operations $\rmr, \rmc, \rmi$.
The permutations $\pi$ and $\mu$ in $\symS_n$ are torically equivalent
if and only if $\pi^\rma$ and $\mu^\rma$ are torically equivalent.
\end{proposition}

\begin{proof}
It clearly suffices to prove this for $\rma \in \{ \rmr, \rmc, \rmi \}$, and since
any such $\rma$ is its own inverse, it suffices to prove just one direction.

We start with $\rma = \rmr$. Since $\pi$ and $\mu$ are torically equivalent
there exists $m \in \dbrac{0,n}$ such that $\pi^\circ \oplus m = \mu^\circ$.
But by Lemma \ref{lem:symmtor}, part (\ref{lem:symmtor-reverse}), we have
\[
(\pi^\rmr)^\circ \oplus m \equiv (\pi^\circ)^\rmr \oplus m \equiv
(\pi^\circ \oplus m)^\rmr \equiv (\mu^\circ)^\rmr \equiv
(\mu^\rmr)^\circ,
\]
so $\pi^\rmr$ and $\mu^\rmr$ are torically equivalent.

Now consider $\rma = \rmc$. Since $\pi$ and $\mu$ are torically equivalent
there exists $m \in \dbrac{0,n}$ such that $\pi^\circ \oplus m = \mu^\circ$.
But by Lemma \ref{lem:symmtor}, part(\ref{lem:symmtor-complement}), we have
\[
(\pi^\rmc)^\circ \ominus m \equiv (\pi^\circ)^\rmc \ominus m \equiv
(\pi^\circ \oplus m)^\rmc \equiv (\mu^\circ)^\rmc \equiv
(\mu^\rmc)^\circ,
\]
so $\pi^\rmr$ and $\mu^\rmr$ are torically equivalent.

Now consider $\rma = \rmi$. Since $\pi$ and $\mu$ are torically equivalent
there exists $m \in \dbrac{0,n}$ such that $\pi^\circ \oplus m = \mu^\circ$.
Let $x$ be the distance (counter-clock-wise) from $n$ to $0$ in $\pi^\circ$.
Then by Lemma \ref{lem:symmtor}, part(\ref{lem:symmtor-inverse}), we have
\[
(\pi^\rmi)^\circ \oplus mx \equiv (\pi^\circ)^\rmi \oplus mx \equiv
(\pi^\circ \oplus m)^\rmi \equiv (\mu^\circ)^\rmi \equiv
(\mu^\rmi)^\circ,
\]
so $\pi^\rmr$ and $\mu^\rmr$ are torically equivalent.
\end{proof}

Then Proposition \ref{prop:basic}
implies the following theorem.

\begin{theorem} \label{thm:toricsymmetry}
Let $\rma$ be a composition of elements from the set
$\{\rmr, \rmc, \rmi\}$. Then
$\rma: \eqAn{\bfp} \to \eqAn{\bfp^\rma}$ is a bijection for all $n$.
\end{theorem}


We further define, for a pattern $\bfp = (p,X,Y)$, of rank $r$,
such that the letter $r$ is at position $\ell$ in $p$ (so $\ell$
equals the last letter of $p^\rmi$)
\[
\bfp \oplus 1
= (p \oplus 1, X \oplus (r+1-\ell), Y \oplus 1)
= (p \oplus 1, X \ominus \ell, Y \oplus 1),
\]
where the addition on the sets is done modulo $r+1$ as usual. Here
is the reason for making this definition:

\begin{example} \label{toricex:firstshift}
Let $\pi = 76128543$ and
\[
\bfp = (3421, \{2,3\}, \{1,2,4\}) = \biv{\niv{1\bl2\bl3}\bl\nivr{4}}{\vinb{3}\bl\vin{4\bl2\bl1}}.
\]
Then $\pi$ has an occurrence of $\bfp$ shown here
$7\ull{6}12\ull{8}\ull{5}\ull{4}3$. Now $\pi \oplus 1 = 65418723$ has
an occurrence of
\[
\bfp \oplus 1 = (3214, \{0,1\}, \{0,2,3\})
= \biv{\nivl{1}\bl\niv{2\bl3\bl4}}{\vinl{3\bl2}\vinb{1\bl4}}
\]
shown here $\ull{6}\ull{5}4\ull{1}8\ull{7}23$.
\end{example}

This example is a special case of the following:

\begin{proposition} \label{prop:newsym}
Let $\bfp = (p,X,Y)$ be a pattern of rank $r$ such that $r \in Y$.
The permutation $\pi$ avoids the pattern $\bfp$ if and only if
the permutation $\pi \oplus 1$ avoids the pattern
$\bfp \oplus 1$.
\end{proposition}

\begin{proof}
Assume the permutation $\pi$ contains the pattern $\bfp$. Then
when we add $1$ modulo $n+1$ the permutation and the pattern are
split up and reassembled at the same spot, so the
permutation $\pi\oplus1$ will contain $\bfp\oplus1$. Reverse this
argument to show the other implication.
\end{proof}

The last Proposition gives a new way to prove Wilf-equivalences for
bivincular patterns:

\begin{corollary} \label{toriccor:Wilfeq}
Let $\bfp = (p,X,Y)$ be a pattern of rank $r$ such that $r \in Y$.
The map
\[
\oplus 1 \colon \symA_n(\bfp) \to \symA_n(\bfp \oplus 1)
\]
is a bijection, and $\eqAn{\bfp} = \eqAn{\bfp \oplus 1}$.
\end{corollary}

This can be iterated as long as we have $r \in Y$ and also combined
with the basic symmetries as the following Example shows.

\begin{example}
\text{}
\begin{enumerate}

\item Consider the two patterns
\[
\bfp = (12, \tom, \{0,2\}) = \biv{\nivl{1}\bl\nivr{2}}{\vinb{1\bl2}}, \qquad
\bfq = (12, \tom, \{0,1\}) = \biv{\nivl{1\bl2}}{\vinb{1\bl2}}
\]
that were shown to be Wilf-equivalent in \cite[subsection 3.4]{P09}.
This can not be done only in terms of the basic symmetries,
but follows from $\bfp \oplus 1 = \bfq$.

\item Consider the two patterns
\[
\bfp = (132, \tom, \{0,1,2\}) = \biv{\nivl{1\bl2\bl3}}{\vinb{1\bl3\bl2}}, \qquad
\bfq = (132, \tom, \{0,2,3\}) = \biv{\nivl{1}\bl\nivr{2\bl3}}{\vinb{1\bl3\bl2}}
\]
that were shown to be Wilf-equivalent in \cite[subsection 5.17]{P09}. This can not be
done only in terms of the basic symmetries, but using the $\oplus 1$ map
we get
\begin{align*}
(\bfp^\rmc \oplus 1)^\rmr &= ((312, \tom, \{1,2,3\}) \oplus 1)^\rmr
= (231, \tom, \{0,2,3\})^\rmr\\
&= (132, \tom, \{0,2,3\}) = \bfq.
\end{align*}

\item Consider the two patterns
\[
\bfp = (123, \{0\}, \{0,3\}) = \biv{\nivl{1}\bl2\bl\nivr{3}}{\vinl{1}\bl\vinb{2\bl3}}, \qquad
\bfq = (123, \{1\}, \{0,1\}) = \biv{\nivl{1\bl2}\bl3}{\vin{1\bl2}\bl\vinb{3}}
\]
that were shown to be Wilf-equivalent in \cite[subsection 5.18]{P09}. This can not be
done only in terms of the basic symmetries, but follows from
$\bfp \oplus 1 = \bfq$.

\item Here is an example showing that we really need to have $r$ in
$Y$. Let $\bfp = \vinb{1}\bl\vin{3\bl2}\bl\vinb{4} = (1324, \{2\}, \tom)$. Then
$\bfp \oplus 1 = (1243, \{3\}, \tom) = \vinb{1\bl2}\bl\vin{4\bl3}$ and
\[
|\symA_6(\bfp)| = 549, \qquad |\symA_6(\bfp \oplus 1)| = 550.
\]

\item The patterns in Example \ref{toricex:firstshift} above are Wilf-equivalent
by Corollary \ref{toriccor:Wilfeq} and this is very likely a new result (since
to my knowledge the question of Wilf-equivalence of bivincular patterns has only
been considered up to length $3$).

\end{enumerate}
\end{example}

Although the $\oplus 1$ map isn't as useful when we don't have $\rank(p)$ in the
vertical set it still gives us:

\begin{lemma} \label{lem:oplus}
Fix an occurrence of a bivincular pattern $\bfp$ in a permutation $\pi$.
Assume $\bfp$ has rank $r$ and let $k$ be the letter in the occurrence
that corresponds to $r$. Then the permutation $\pi \ominus k$ contains
$\bfp \oplus 1$.
\end{lemma}

\begin{theorem} \label{thm:substable}
For any pattern $\bfp$ we have $\eqAn{\bfp} \subseteq \Aeqn{\bfp}$
\end{theorem}

\begin{proof}
Assume $\pi$ is not in the set on the right, so $\pi$ contains a pattern
$\bfp'$ that is torically equivalent to $\bfp$.
We can assume that $\bfp = \bfp' \oplus 1$. By Lemma \ref{lem:oplus}
we see that a permutation $\pi'$ that is torically equivalent to $\pi$
contains $\bfp$. This shows that $\pi$ is not in the set on the left.
\end{proof}

This theorem will be crucial when we prove Theorems \ref{toricprop:213},
\ref{thm:totient}, \ref{thm:divisors} below.

\subsection{Number and size of classes}
The number of equivalence classes in $\symS_n$ are
\[
1, 2, 3, 8, 24, 108, 640, 4492,\dotsc \qquad n = 0,1,2,3, \dotsc.
\]
This is A002619: Number of $2$-colored patterns on an $(n+1) \times (n+1)$ board.
Instead of proving that this is indeed true, let us prove a stronger result,
that enumerates the number of classes of a particular size:

For each $n$ we can write $n! = \sum (\text{number of classes of size $i$}) \cdot i$. For
$n = 1, \dotsc, 8$ this gives us the following equations,
\begin{align*}
1! &= 1 \cdot 1, \\
2! &= 1 \cdot 2, \\
3! &= 1 \cdot 2 + 4 \cdot 1, \\
4! &= 1 \cdot 4 + 5 \cdot 4, \\
5! &= 1 \cdot 2 + 2 \cdot 2 + 3 \cdot 2 + 6 \cdot 18, \\
6! &= 1 \cdot 6 + 7 \cdot 102, \\
7! &= 1 \cdot 4 + 2 \cdot 2 + 4 \cdot 10 + 8 \cdot 624, \\
8! &= 1 \cdot 6 + 3 \cdot 10 + 9 \cdot 4476.
\end{align*}

The following theorem gives us a way to explain the number that appear above, i.e., how
many classes of size $i$ there are.
\begin{theorem} \label{toricthm:fromSteggall}
The number of classes in $\symS_n$ of size $k$ is
\[
\frac{1}{(n+1)k}\sum_{d|k} \mu(d) U(n+1,\frac{k}{d})
\]
if $k|(n+1)$ but zero otherwise. Here $U(n+1,\ell) =
\phi(\frac{n+1}{l}) \left( \frac{n+1}{l} \right)^\ell \ell!$
\end{theorem}

Before proving this, say that a a permutation $\pi$ in $\symS_n$
is \emph{cyclically invariant under $\oplus k$} if $\pi$ and
$\pi\oplus k$ are the same when considered as circular permutations.

\begin{proof}
Everything needed to prove this is in \cite{St07}, except the last step where
we apply the M\"obius inversion formula. We also adapt Steggall's
arguments to our notation. Steggall showed that the number of permutations in $\symS_n$ that
are cyclically invariant under $\oplus k$ is
\[
U(n,k) = \begin{cases}

\phi(\frac{n}{k}) \left( \frac{n}{k} \right)^k k!      & \text{if $k | n$}, \\

0      & \text{otherwise}.
\end{cases}
\]
This number includes permutations that are cyclically invariant under $\oplus d$,
where $d$ is factor of $k$. Let $u(n,d)$ count the permutations that are cyclicly
invariant under $d$, but no smaller factor. We therefore have
\[
U(n,k) = \sum_{d|k} u(n,d).
\]
Steggall inverts this formula, writing $u(n,d)$ in terms of $U(n,k)$ for some special
cases. We can combine
these by applying the M\"obius inversion formula
\[
u(n,k) = \sum_{d|k} \mu(d) U(n,\frac{k}{d}),
\]
where $\mu$ is the M\"obius function. We now have a formula that counts the number
of permutations in $\symS_n$ that are cyclically invariant under $\oplus k$
and no smaller factor. Turning these into cyclic permutations their number is
$u(n,k)/n$ and the action of $\oplus 1$ breaks this set into orbits of size $k$,
each one corresponding to a toric class of size $k$ in $\symS_{n-1}$.
\end{proof}

%
%
%
%

Here is the array we get by calculating a few values of the number of classes
in $\symS_n$, starting from $n = 0$:
\[
\begin{tabular}{ccccccccc}
1 \\
1 & 0 \\
2 & 0 & 0 \\
2 & 0 & 0 & 1 \\
4 & 0 & 0 & 0 & 4 \\
2 & 2 & 2 & 0 & 0 & 18 \\
6 & 0 & 0 & 0 & 0 & 0 & 102 \\
4 & 2 & 0 & 10 & 0 & 0 & 0 & 624 \\
6 & 0 & 10 & 0 & 0 & 0 & 0 & 0 & 4476
\end{tabular}
\]

It is easy to verify that for $k = 1$ we get $\phi(n+1)$ classes of
size $1$ from the formula in Theorem {toricthm:fromSteggall}.
The permutations that lie in these one element classes will be
important later; see Lemma \ref{lem:toricclasssize1} below, which gives another
proof of the enumeration of these classes.

I should note that in \cite{MR2028287} toric classes appear in a disguise
as \emph{$(n+1)$-orbits}. The author shows that the total number of these
orbits is
\[
\frac{1}{(n+1)^2}\sum_{kp=n+1} \phi(p)^2 k! p^k,
\]
by the orbit counting lemma. This gives A002619 in OEIS, and is equal to
the row sums of the array above.

\subsection{Pattern avoidance}

\subsubsection*{Patterns and modular sequences}

A \emph{modular $k$ sequence} is sequence of letters
$\ell_1, \ell_2, \dots, \ell_k$ such that $\ell_{i+1} = \ell_i + 1 \bmod{(n+1)}$.

\begin{proposition}
If $\bfp = \biv{\nivls{1}}{\vinls{1}}$ 
then the set $\eqAn{\bfp}$ is equinumerous
with the set $C_\dbrac{0,n}'$ of circular permutations of $\dbrac{0,n}$ that
have no modular $2$-sequences, i.e., $i$ is never followed by $i+1 \bmod{(n+1)}$.
\end{proposition}

Modular $2$-sequences are sometimes called \emph{successor pairs}.
The enumeration corresponds to sequence A000757:
\[
0, 1, 1, 8, 36, 229, 1625, 13208,\dotsc, \qquad n = 1,2,3, \dotsc.
\]

\begin{proof}
We will construct a bijection between
$\eqAn{\bfp}$ and the set of circular permutations of $\dbrac{0,n}$
without successor pairs. We do this as follows:
\[
f \colon \eqAn{\bfp} \to C_\dbrac{0,n}'
\]
We let $f(\pi) = \pi^\circ = 0\pi$. We first need to show that the image of $f$ is
in the set on the right. So assume $f(\pi)$ has a successor pair, $i(i+1) \bmod{(n+1)}$.
Then $(f(\pi) \oplus (n+1-i))_\circ$ is torically equivalent to $\pi$ and starts with
$1$, but that is impossible, since all permutations torically equivalent to $\pi$
avoid $\bfp = (1, \{0\}, \{0\})$.

Now there is a map in the other direction $g(\lambda) = \lambda_\circ$. A similar
argument to the one above shows that the image of $g$ is actually in the set
on the left. The maps are obviously inverses of each other so we are done.
\end{proof}


 

\begin{proposition}
If $\bfp = \biv{\nivls{1\bl2}}{\vinls{1\bl2}}$ 
then the set $\eqAn{\bfp}$ is equinumerous
with the set $C_\dbrac{0,n}''$ of circular permutations of $\dbrac{0,n}$ that
have no modular $3$-sequences, i.e., $i$ is never followed by $i+1$ and $i+2$
$\bmod{(n+1)}$.
\end{proposition}

The enumeration is
\[
|\eqAn{\bfp}| = 1, 1, 5, 18, 95, 600, 4307, 35168, \qquad (n = 1,2,3,4,5,6,7,8).
\]
This corresponds to sequence A165962.

\begin{proof}
We will construct a bijection between
$\eqAn{\bfp}$ and the set of circular permutations of $\dbrac{0,n}$
without modular $3$-sequences. We do this as follows:
\[
f \colon \eqAn{\bfp} \to C_\dbrac{0,n}''
\]
We let $f(\pi) = \pi^\circ = 0\pi$. We first need to show that the image of $f$ is
in the set on the right. So assume $f(\pi)$ has a modular $3$-sequence,
$i(i+1)(i+2) \bmod{(n+1)}$.
Then $(f(\pi) \oplus (n+1-i))_\circ$ is torically equivalent to $\pi$ and starts with
$12$, but that is impossible, since all permutations torically equivalent to $\pi$
avoid $\bfp$.

Now there is a map in the other direction $g(\lambda) = \lambda_\circ$. A similar
argument to the one above shows that the image of $g$ is actually in the set
on the left. The maps are obviously inverses of each other so we are done.
\end{proof}

It then turns out that you can keep going (the proof is an easy generalization
of the proofs above and we omit it):
\begin{proposition}
For $k \geq 1$, let $\bfp = \biv{\nivls{1\bl2\dotsm k}}{\vinls{1\bl2\dotsm k}}$.
The set $\eqAn{\bfp}$ is equinumerous with the set of
circular permutations of $\dbrac{0,n}$ that have no modular $(n+1)$-sequences.
\end{proposition}

\subsubsection*{Patterns with connections to number theory}

Before we look at the next pattern we have some definitions to make.

\begin{definition} \label{toricdef:naturaldivisor}
Let $n \geq 1$ be fixed.
\begin{enumerate}

\item  For any $k$ coprime to $n+1$ we define a permutation $\nu_{k,n}$
first by constructing a circular permutation $\lambda_{k,n}$ of $\dbrac{0,n}$ as follows:
Place $0$ anywhere, then place $1$ by moving $k$ steps from $0$ (so there are $k-1$
empty positions between $0$ and $1$), then place $2$ by moving $k$ steps from $1$
and keep going until you place $n$. Then define $\nu_{k,n} = (\lambda_{k,n})_\circ$.
We call the permutation constructed in this way the \emph{natural permutation}
(\emph{corresponding to $k$})
in $\symS_n$.

\item If $k$ is a divisor of $n$ we write $\delta_{k|n} = \nu_{k,n}$ and call
$\delta_{k|n}$ the \emph{divisor permutation} (\emph{corresponding to $k|n$})
in $\symS_n$.

\end{enumerate}
\end{definition}
The condition that $k$ be coprime to $n+1$ is a necessary and sufficient condition
for constructing $\nu_{k,n}$.
The reason for calling these permutations \emph{natural}
is that the permutation $\nu_{k,n}$ behaves like the natural number $k$ when multiplied
with other natural permutations (see Proposition \ref{prop:grouphom} below).
The construction of the natural permutations is slightly reminiscent of the Josephus problem,
see e.g., \cite{S02}.
\begin{example}
Let $n = 6$. The coprime integers to $n+1 = 7$ are $1,2,3,4,5$ and $6$.
We construct $\lambda_{1,6}$ as follows:
\[
0\underline{\phantom{123456}} = 01\underline{\phantom{23456}}
= 012\underline{\phantom{3456}} = \dotsm = 0123456,
\]
so $\nu_{1,6} = 123456 = \delta_{1|6}$. Next we construct $\lambda_{2,6}$:
\[
0\underline{\phantom{415263}} =
0\underline{\phantom{4}}1\underline{\phantom{5263}} =
0\underline{\phantom{4}}1\underline{\phantom{5}}2\underline{\phantom{63}} =
0\underline{\phantom{4}}1\underline{\phantom{5}}2\underline{\phantom{6}}3 =
041\underline{\phantom{5}}2\underline{\phantom{6}}3 =
04152\underline{\phantom{6}}3 =
0415263,
\]
so $\nu_{2,6} = 415263 = \delta_{2|6}$. Similarly we get
\begin{align*}
\nu_{3,6} &= 531642 = \delta_{3|6}, \\
\nu_{4,6} &= 246135, \\
\nu_{5,6} &= 362514, \\
\nu_{6,6} &= 123456 = \delta_{6|6}.
\end{align*}
\end{example}
Below we give list of all the natural permutations in $\symS_1$ to $\symS_{10}$

\noindent
Note that by definition the location of $1$ in $\nu_{k,n}$ and $\delta_{k|n}$
is $k$, or equivalently, the first letter of inverse of these permutations is $k$.
They also have other good properties:

\begin{lemma} \label{lem:incr}
\text{}
\begin{enumerate}

\item For a natural permutation $\nu = \nu_{k,n}$ the difference between
subsequent letters,
\[
j = \nu(\ell+1) - \nu(\ell)
\]
is independent of $\ell$ and equals the first letter of $\nu$.

\item  If $\nu$ is a divisor permutation then $j = - n/k \bmod{(n+1)}$.
In general this difference is the smallest positive integer $j$
such that $kj = 1 \bmod{(n+1)}$.

\end{enumerate}
\end{lemma}

\begin{proof}
\begin{enumerate}

\item The letter that lands at $\ell$ equals
\[
\frac{\ell + (n+1)s_\ell}{k}
\]
where $s_\ell$ is chosen as the least integer to make this an integer.
The letter that lands at $\ell+1$ equals
\[
\frac{\ell+1 + (n+1)s_{\ell+1}}{k}
\]
where $s_{\ell+1}$ is chosen as the least integer to make this an integer.
The difference is
\[
\frac{1-(n+1)(s_{\ell+1}-s_\ell)}{k}
\]
which is independent of $\ell$. This difference must equal the first letter
since if we add $0$ at the front the difference is the same as the first element.

\item This can be done by examining the difference above: If $k | n$ then
we choose $s_{\ell+1} = s_\ell - 1$ and the difference becomes $n/k$. The
general case is similar.

\end{enumerate}
\end{proof}

In view of the last Lemma we will call the first letter of a natural
permutation $\nu$ the \emph{increment} of the permutation and denote
it by $j_\nu$ or just $j$ if there are no other natural permutations around.

\begin{lemma} \label{lem:revnat}
\text{}
\begin{enumerate}

\item $\nu_{k,n}^\rmr = \nu_{n+1-k,n} = \nu_{k,n}^\rmc$.
\item $\nu_{k,n}^\rmi = \nu_{j,n}$, where $j$ is the increment of $\nu_{k,n}$.

\end{enumerate}

\end{lemma}

\begin{proof}
This follows directly from Definition \ref{toricdef:naturaldivisor} and Lemma \ref{lem:incr}.
Note that $k$ is coprime to $n+1$ if and only if $n+1-k$ is coprime
to $n+1$.
\end{proof}

\begin{lemma}
Let $\nu = \nu_{k,n}$ be a natural permutation.
\begin{enumerate}

\item $\nu_1 + \nu_n = n+1$.

\item If $k | n$ then $\nu_n = n/k$, and in general $\nu_n = -j \bmod{(n+1)}$, where $j$ is
the increment of $\nu$.

\end{enumerate}
\end{lemma}

\begin{proof}
Note that $\nu_n = (\nu^\rmr)_1$ and the first letter of $\nu^r$ is
the smallest positive solution of $(n+1-k)x \equiv 1 \bmod{(n+1)}$. Clearly
this is $n+1-j$.
\end{proof}

\begin{lemma} \label{lem:divpermssubseq}
If $\delta_{k|n}$ is a divisor permutation then it consists of $k$ increasing
subsequences of length $n/k$. These sequences are 
\begin{align*}
& 1,2,\dotsc,n/k, \\
& n/k+1,n/k+2, \dotsc, 2n/k, \\
& \dotsc, \\
& n-n/k+1, n-n/k+2,\dotsc, n.
\end{align*}
They lie inside the permutation in such a way that the $\ell$-th one
starts at $k-\ell+1$ and is placed at locations $k-\ell+1+dk$, where
$d = 0,\dotsc,n/k-1$.
\end{lemma}

\begin{proof}
This is obvious from the construction of the divisor permutations.
\end{proof}

\begin{example}
Consider the divisor permutation 
\[
\delta_{3|12} = 9\ull{5}1(10)\ull{6}2(11)\ull{7}3(12)\ull{8}4,  
\]
it consists of $3$ increasing subsequences of lenght $12/3=4$. We underline
the second one.
\end{example}

\begin{lemma} \label{lem:toricclasssize1}
The natural permutations are the permutations with a toric class
of size $1$.
\end{lemma}

\begin{proof}
It is easy to see from the construction of the natural permutations that
they have a toric class of size $1$. So assume $\pi$ is a permutation
with toric class size $1$, which is equivalent to $\pi \oplus 1 = \pi$,
which is equivalent to $\pi^\circ \oplus 1 \equiv \pi^\circ$. Now let
$\ell$ be the distance between $0$ and $1$ in $\pi^\circ$. This becomes
the distance between $1$ and $2$ in $\pi^\circ \oplus 1 = \pi^\circ$.
By iterating this we have shown that the distance between $k$ and $k+1
\bmod{(n+1)}$ is fixed. This implies that $\pi$ is a natural permutation.
\end{proof}

\begin{proposition} \label{prop:grouphom}
The natural permutations multiply like the natural numbers modulo $n+1$, i.e.,
$\nu{k,n} \circ \nu_{k',n} = \nu_{kk',n}$, where $kk'$
has been reduced modulo $n+1$.
\end{proposition}

\begin{proof}
This is easy to see using the increment.
\end{proof}

\begin{example}
For example
\[
\delta_{2|6} \circ \delta_{2|6} = \nu_{4,6} \quad \textrm{ and } \quad
\nu_{4,6} \circ \nu_{5,6} = \delta_{6|6}.
\]
\end{example}

\begin{proposition} \label{toricprop:213}
The set $\eqAn{213}$ consists of the identity and
its reverse.
\end{proposition}

\begin{proof}
It is clear that the identity and its reverse are in this set.
By Theorem \ref{thm:substable} a permutation in this set must avoid
the classical patterns $2\dash1\dash3$, $1\dash3\dash2$, $3\dash1\dash2$
and $2\dash3\dash1$,
so the only possibilities are the identity and its reverse.
\end{proof}

\begin{theorem} \label{thm:totient}
Let $\bfp = \biv{\nivs{1\bl2}\bl\nivrs{3}}{\vinbs{2\bl1\bl3}}$.
The set $\eqAn{\bfp}$ consists of the natural permutations
in $\symS_n$.
In particular
\[
|\eqAn{\bfp}| = \phi(n+1),
\]
where $\phi$ is Euler's totient function,
and the map $\nu_{k,n} \mapsto k,  \eqAn{\bfp} \to U_{n+1}$
is a homomorphism of groups, where
$U_{n+1}$ is the group of units in $\ZZ/(n+1)\ZZ$.
\end{theorem}

\begin{proof}
We start by showing that the natural permutations avoid this pattern.
Fix such a permutation $\nu = \nu_{k,n}$ with $k$ coprime to $n+1$.
It suffices to show that given an integer $2 \leq \ell \leq n-1$ such
that $\ell-1$ occurs after $\ell$ in the permutation, then $n$ appears
before $\ell-1$. Let $j$ be the increment of $\nu$ so we can write
\[
\nu = j,2j,3j,\dotsc,mj,\dotsc,(m+h)j,\dots,nj \quad \bmod{(n+1)}
\]
where $mj = \ell \bmod{(n+1)}$ and $(m+h)j = \ell-1 \bmod{(n+1)}$. Then
we must have $hj = {-1} \bmod{(n+1)}$ which implies $hj = n \bmod{(n+1)}$ so
$n$ is at location $h$ in $\nu$ and therefore appears before $\ell-1$.

We must now show that any permutation that avoids this pattern must
be a natural permutation. This relies on Lemmas \ref{lem:1+x} and
\ref{lem:z+x}. They show that $\pi = x (2x) (3x) \dotsm$
and is therefore natural.
\end{proof}

\begin{lemma} \label{lem:1+x}
If $\pi$ is in the set $\eqAn{\biv{\nivs{1\bl2}\bl\nivrs{3}}{\vinbs{2\bl1\bl3}}}$ and
$\pi = x \dotsm 1y \dotsm$ then $y = 1+x \bmod{(n+1)}$.
\end{lemma}

\begin{proof}
Note that if $x = n$ then $\pi = \id^\rmr$, so we can assume that
$x \neq n$.
Assume $y \neq 1+x \bmod{(n+1)}$. Then $1+x$ must appear after
$y$ in $\pi$, otherwise we have the pattern $\bfp \oplus 3$. 

\begin{figure}[htbp]
\begin{center}
\begin{tikzpicture}[auto,bend right]
\node (0) at (90:1) {$0$};
\node (x) at (76:1) {$x$};
\node (1) at (-30:1) {$1$};
\node (y) at (-54:1) {$y$};
\foreach \p in {-78,-102,-126,-174,-198,...,-264,-312,-336,-360}
	\node (\p) at (\p:1) {$\cdot$};
\node (1x) at (-150:1) {$1+x$};
\end{tikzpicture}
\caption{The circular permutation $\pi^\circ$.}
\end{center}
\end{figure}
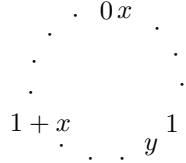

\noindent
Consider what values $y$ could take. It can not be $2$ since that would
produce the pattern in $\bfp \oplus 1$ in $\pi \ominus 1$. It can not be $3$
since that would mean $2$ was between $x$ and $1$ in $\pi$ and $\pi \ominus 2$
would contain $\bfp \oplus 1$. Similarly for any $k < x$ we can show $y \neq k$. Therefore
we must have $y > 1+x$. If $y = 2+x$ then $1, 2+x, 1+x$ in $\pi$ give the pattern
$\bfp \oplus 1$. The letter $2+x$ must therefore be between $1+x$ and $0$ in
$\pi^\circ$ above. We can now see that $y$ can not be $3+x$ either. By induction
we can show that $y$ can not be any letter larger than $1+x$. This gives us a
contradiction so we must have $y = 1+x$.
\end{proof}

\begin{lemma} \label{lem:z+x}
If $\pi$ is in the set $\eqAn{\biv{\nivs{1\bl2}\bl\nivrs{3}}{\vinbs{2\bl1\bl3}}}$ and
$\pi = x \dotsm zy \dotsm$ then $y = z+x \bmod{(n+1)}$.
\end{lemma}

\begin{proof}
We do this by induction on $z$. This is known for $z= 1$ from Lemma
\ref{lem:1+x}. Assume it is true for $z-1$ and consider
$\pi = x \dotsm zy \dotsm$. Then $\pi \ominus 1 = \dotsm n (x-1) \dotsm (z-1)(y-1) \dotsm$.
This permutation must then start with $y-z$. So in $\pi$ we have
$\dotsm 1 (y-z+1) \dotsm$. But $\pi$ starts with $x$ so we get
$y-z+1 = 1+x$ or $y = z+x$ as we wanted.
\end{proof}


\begin{theorem} \label{thm:divisors}
Let $\bfp = \biv{\nivs{1\bl2}\bl\nivbs{3}}{\vinbs{2\bl1\bl3}}$.
The set $\eqAn{\bfp}$ consists of the divisor permutations
in $\symS_n$. In particular
\[
|\eqAn{\bfp}| = d(n),
\]
where $d(n)$ counts the number of divisors in $n$.
\end{theorem}

\begin{proof}
If a divisor permutation contains the pattern $\biv{\niv{1\bl2}}{\vinb{2\bl1}}$ then
the letter corresponding to $1$ must be at the end of a subsequence and the letter
corresponding to $2$ at the start of the next subsequence (see Lemma \ref{lem:divpermssubseq}).
But everything after the end of a subsequence is smaller than the elements in
the subsequence so the pattern can not be completed to $\biv{\niv{1\bl2}\bl3}{\vinb{2\bl1\bl3}}$.
Since a divisor permutation is alone in its toric class we are done.

Now suppose we have a permutation that is in the set on the left. Then it must be
a natural permutation, $\nu_{k,n}$, by Theorem \ref{thm:totient}. To finish the
proof we need to show that $k|n$. Suppose not. Then the first subsequence in $\nu_{k,n}$
consisting of $1,2,3,\dotsc \ell$ does not terminate at the last position of the permutation,
meaning that something larger than $\ell$ appears there. This clearly gives us an
occurrence of the pattern, which is a contradiction.
\end{proof}

Recall that the location of $1$ in $\nu_{k,n}$ and $\delta_{k|n}$
is $k$, or equivalently, the first letter of inverse of these permutations is $k$. 
This property combined with Theorem \ref{thm:divisors}
gives us a way to write the \emph{sum-of-divisors} function $\sigma(n)$ as
\[
\sigma(n) =
\sum_{\delta \in \eqAn{\biv{\nivs{1\bl2}\bl\nivbs{3}}{\vinbs{2\bl1\bl3}}}}
(\textrm{location of $1$ in $\delta$}).
\]

Now consider the following theorem due to \cite{MR774171}.

\begin{theorem}[Robin's theorem]
Let $\sigma(n)$ denote the sum of the divisors of $n$. The Riemann Hypothesis is true
if and only if
\[
\sigma(n) < e^\gamma \log \log n,
\]
holds for all $n \geq 5041$. Here $\gamma$ is the Euler-Mascheroni constant.
\end{theorem}

\noindent
This allows us to state the Riemann Hypothesis in terms of pattern avoidance:
\begin{conjecture}[Equivalent to RH]
\label{conj:eqRH}
The inequality
\[
\sum_{\delta \in \eqAn{\biv{\nivs{1\bl2}\bl\nivbs{3}}{\vinbs{2\bl1\bl3}}}}
(\textrm{location of $1$ in $\delta$}) < e^\gamma \log \log n,
\]
holds for all $n \geq 5041$.
\end{conjecture}
The largest known $n$ for which the inequality in Robin's theorem is violated is $5040$,
so it \emph{suffices} to start exploring in $\symS_{5041}$. I should mention that permutations
have been shown before to have connections with the Riemann Hypothesis, for example using the
Redheffer matrix in \cite{MR2120104}, probabilistic methods in \cite{MR1694204} (see also \cite{Stopple}),
and group theory in \cite{MR960551}.


\subsubsection*{Natural permutations in $\symS_1$ to $\symS_{10}$}

In $\symS_1$ we have
\begin{align*}
\nu_{1,1} &= 1 = \delta_{1|1}
\end{align*}

In $\symS_2$ we have
\begin{align*}
\nu_{1,2} &= 12 = \delta_{1|2} \\
\nu_{2,2} &= 21 = \delta_{2|2}
\end{align*}

In $\symS_3$ we have
\begin{align*}
\nu_{1,3} &= 123 = \delta_{1|3} \\
\nu_{3,3} &= 321 = \delta_{3|3}
\end{align*}

In $\symS_4$ we have
\begin{align*}
\nu_{1,4} &= 1234 = \delta_{1|4} \\
\nu_{2,4} &= 3142 = \delta_{2|4} \\
\nu_{3,4} &= 2413 \\
\nu_{4,4} &= 4321 = \delta_{4|4}
\end{align*}

In $\symS_5$ we have
\begin{align*}
\nu_{1,5} &= 12345 = \delta_{1|5} \\
\nu_{5,5} &= 54321 = \delta_{5|5}
\end{align*}

In $\symS_6$ we have
\begin{align*}
\nu_{1,6} &= 123456 = \delta_{1|6} \\
\nu_{2,6} &= 415263 = \delta_{2|6} \\
\nu_{3,6} &= 531642 = \delta_{3|6} \\
\nu_{4,6} &= 246135 \\
\nu_{5,6} &= 362514 \\
\nu_{6,6} &= 654321 = \delta_{6|6}
\end{align*}

In $\symS_7$ we have
\begin{align*}
\nu_{1,7} &= 1234567 = \delta_{1|7} \\
\nu_{3,7} &= 3614725 \\
\nu_{5,7} &= 5274163 \\
\nu_{7,7} &= 7654321 = \delta_{7|7}
\end{align*}

In $\symS_8$ we have
\begin{align*}
\nu_{1,8} &= 12345678 = \delta_{1|8} \\
\nu_{2,8} &= 51627384 = \delta_{2|8} \\
\nu_{4,8} &= 75318642 = \delta_{4|8} \\
\nu_{5,8} &= 24681357 \\
\nu_{7,8} &= 48372615 \\
\nu_{8,8} &= 87654321 = \delta_{8|8}
\end{align*}

In $\symS_9$ we have
\begin{align*}
\nu_{1,9} &= 123456789 = \delta_{1|9} \\
\nu_{3,9} &= 741852963 = \delta_{3|9} \\
\nu_{7,9} &= 369258147 \\
\nu_{9,9} &= 987654321 = \delta_{9|9}
\end{align*}

In $\symS_{10}$ we have
\begin{align*}
\nu_{1,10} &= 123456789(10) = \delta_{1|10} \\
\nu_{2,10} &= 61728394(10)5 = \delta_{2|10} \\
\nu_{3,10} &= 4815926(10)37 \\
\nu_{4,10} &= 369147(10)258 \\
\nu_{5,10} &= 97531(10)8642 = \delta_{5|10} \\
\nu_{6,10} &= 2468(10)13579 \\
\nu_{7,10} &= 852(10)741963 \\
\nu_{8,10} &= 73(10)6295184 \\
\nu_{9,10} &= 5(10)49382716 \\
\nu_{10,10} &= (10)987654321 = \delta_{10|10}
\end{align*}

%

\section{Other equivalence relations} \label{sec:other}

I have looked at many other equivalence relations besides the one we study above.
Below are just two conjectures concerning two of them.

\begin{conjecture}
Regard two permutations as equivalent if their descents appear at the same positions.
The pattern $\bfp = 321$ gives
\begin{align*}
\eqAn{\bfp} &= \text{ Grassmannian permutations} \\
|\eqAn{\bfp}| &= 1, 2, 5, 12, 27, 58, 121, 248, 503, \dotsc, \qquad n = 1,2,3,\dotsc \\
 &= 2^n-n.
\end{align*}
This is A000325.
\end{conjecture}

Recall that a \emph{Grassmannian permutation} has at most one descent

\begin{conjecture}
Regard two permutations as equivalent if they give the same \emph{volume partition},
see \cite{Y10}.
The pattern $\bfp = \biv{\nivls{1\bl2}\bl\vinbs{3}}{\vinbs{1\bl2\bl3}}$ gives
\begin{align*}
|\eqAn{\bfp}| &= 1,2,5,18,84,480,3240,25200, \dotsc, \qquad n = 1,2,3,\dotsc 
\end{align*}
which is (from $n = 2$)
A038720: Next-to-last diagonal of A038719, which counts chains
in a specific poset.
\end{conjecture}


\bibliographystyle{abbrvnat}
\bibliography{RefsEq_avoiders}
\label{sec:biblio}

\end{document}

%% file: PreambleHR2010.tex
\usepackage{fancyvrb}

\usepackage{amsthm,amsmath,amssymb,amsfonts}

\usepackage{mathrsfs}

\usepackage{stmaryrd}

\usepackage{paralist}

\usepackage{lscape}

\usepackage{tikz}
\usetikzlibrary{matrix,arrows}
\usetikzlibrary{positioning}
\usetikzlibrary{automata}
\usetikzlibrary{fit}
\usetikzlibrary{shapes.geometric}

\usepackage{youngtab}

\usepackage{multirow}		

\usepackage{marginnote}

\usepackage{enumerate}

\usepackage{mathtools}

\newcommand{\sep}{\,|\,} 							

\newcommand{\tom}{\varnothing}					

\newcommand{\NN}{\mathbb{N}}						
\newcommand{\ZZ}{\mathbb{Z}}						

\newcommand{\dbrac}[1]{{\llbracket #1 \rrbracket}} 	

\newcommand{\dash}{{-}}							

\newcommand{\symS}{{\mathfrak S}}					
\newcommand{\symM}{{\mathfrak M}}
\newcommand{\symA}{{\mathfrak A}}

\newcommand{\eqA}[2]{\widetilde{\symA}_{#1}\left(#2\right)}	
\newcommand{\eqAn}[1]{\eqA{n}{#1}}				

\newcommand{\eqM}[2]{\widetilde{\symM}_{#1}\left(#2\right)}	
\newcommand{\eqMn}[1]{\eqM{n}{#1}}				

\newcommand{\Aeq}[2]{\symA_{#1}\left(\widetilde{#2}\right)}		
\newcommand{\Aeqn}[1]{\Aeq{n}{#1}}				

\newcommand{\Meq}[2]{\symM_{#1}\left(\tilde{#2}\right)}		
\newcommand{\Meqn}[1]{\Meq{n}{#1}}				

\newcommand{\bfp}{\mathbf{p}}
\newcommand{\bfq}{\mathbf{q}}

\newcommand{\rma}{\mathrm{a}}	
\newcommand{\rmc}{\mathrm{c}}	
\newcommand{\rmi}{\mathrm{i}}	
\newcommand{\rmr}{\mathrm{r}}	
\DeclareMathOperator{\rank}{rank}	

\newcommand{\ull}[1]{\underline{\underline{#1}}}

\newcommand{\biv}[2]{\substack{\displaystyle{#1} \\ \displaystyle{#2}}}

\def\vin#1{{%
  	\vbox{
    		\hbox{${#1}$}%
		\kern 1pt 
		\hrule height .7pt 
  	}
}}

\newcommand{\vins}[1]{\vin{\scriptstyle #1}}

\def\niv#1{{%
  	\vbox{
		\hrule height .7pt 
    		\kern 1pt 
    		\hbox{${#1}$}%
  	}
}}

\newcommand{\nivs}[1]{\niv{\scriptstyle #1}}

\def\vinl#1{{%
	\vrule width .7pt
  	\vbox{
    		\hbox{\kern0.5pt${#1}$}%
		\kern 1pt 
		\hrule height .7pt
  	}
	\kern -0.5pt
}}

\newcommand{\vinls}[1]{\vinl{\scriptstyle #1}}

\def\nivl#1{{%
  	\vrule width .7pt
  	\vbox{ 
  		\hrule height .7pt 
    		\kern 1pt
    		\hbox{\kern 0.5pt${#1}$}%
  	}
	\kern -0.5pt 
}}

\newcommand{\nivls}[1]{\nivl{\scriptstyle #1}}

\def\vinr#1{{%
  	\vbox{
    		\hbox{${#1}$\kern 0.5pt}%
		\kern 1pt 
		\hrule height .7pt
  	}
	\kern.05pt
	\vrule width .7pt
}}

\newcommand{\vinrs}[1]{\vinr{\scriptstyle #1}}

\def\nivr#1{{%
  	\vbox{
		\hrule height .7pt 
    		\kern 1pt 
    		\hbox{${#1}$\kern 0.5pt}%
  	}
	\kern.05pt 
  	\vrule width .7pt
}}

\newcommand{\nivrs}[1]{\nivr{\scriptstyle #1}}

\def\vinb#1{{%
  	\vbox{
    		\hbox{${#1}$}%
		\kern 1.72pt 
  	}
}}

\newcommand{\vinbs}[1]{\vinb{\scriptstyle #1}}
\newcommand{\nivbs}[1]{{\scriptstyle #1}}

\newcommand{\bl}{{\kern 1.75pt}}

\newcommand{\id}{{\mathrm{id}}}

%% file: EnglishEnvironments.tex
{  \theoremstyle{plain} 
   \newtheorem*{named-thm}{}

   \newtheorem*{exercise*}{Exercise}
   \newtheorem{theorem}{Theorem}[section]
   \newtheorem{proposition}[theorem]{Proposition} 
   \newtheorem{lemma}[theorem]{Lemma}
   \newtheorem{corollary}[theorem]{Corollary}
   
   \newtheorem{example}[theorem]{Example}
   \newtheorem{conjecture}[theorem]{Conjecture}

   \newtheorem{openproblem}{Open Problem}
}
{\theoremstyle{definition}
   \newtheorem{definition}[theorem]{Definition}
   \newtheorem{defn-prop}[theorem]{Skilgreining-Fullyrðing}

}